\documentclass[a4paper,fleqn]{cas-sc}

\usepackage[utf8]{inputenc}
\usepackage[T1]{fontenc}
\usepackage{amsmath,amssymb,amsthm,booktabs,graphicx,microtype}
\usepackage[numbers,sort&compress]{natbib}
\usepackage{doi,xcolor,placeins,float}
\usepackage{array,needspace,tabularx,longtable}
\usepackage{etoolbox}
\newif\ifshowedits
\showeditsfalse
\newcommand{\ed}{\ifshowedits\color{red}\fi}
\newcommand{\edtitle}[1]{\ifshowedits\textcolor{red}{#1}\else#1\fi}
\preto\section{\Needspace{7\baselineskip}}
\preto\subsection{\Needspace{6\baselineskip}}
\preto\subsubsection{\Needspace{4\baselineskip}}
\theoremstyle{plain}
\newtheorem{proposition}{Proposition}[section]
\newtheorem{lemma}[proposition]{Lemma}
\newtheorem{corollary}[proposition]{Corollary}
\theoremstyle{definition}
\newtheorem{definition}[proposition]{Definition}
\theoremstyle{remark}
\newtheorem{remark}[proposition]{Remark}
\numberwithin{equation}{section}

\newcommand{\E}{\mathbb{E}}
\newcommand{\Var}{\operatorname{Var}}
\newcommand{\R}{\mathbb{R}}
\newcommand{\sgn}{\operatorname{sgn}}
\newcommand{\ind}{\mathbf{1}}
\newcommand{\Fil}{\mathcal{F}}
\newcommand{\refl}{\textnormal{\textsc{reflect}}}
\newcommand{\swi}{\textnormal{\textsc{switch}}}
\newcommand{\wit}{\textnormal{\textsc{within}}}
\graphicspath{{./}{figures/}}
\makeatletter
\renewcommand\paragraph{\@startsection{paragraph}{4}{\z@}{1.2ex \@plus .4ex \@minus .2ex}{-0.6em}{\normalfont\normalsize\bfseries}}
\makeatother

\graphicspath{{figures/}}

\ExplSyntaxOn
\cs_set:Npn \__first_footerline:
  {
    \group_begin:
    \small\sffamily
    \ifnum\theblind>0\relax\else\__short_authors: :~\fi
    {\rmfamily\itshape Preprint}
    \group_end:
  }
\ExplSyntaxOff

\begin{document}
\let\WriteBookmarks\relax
\def\floatpagepagefraction{1}
\def\textpagefraction{.001}

\shorttitle{Sign-switched antithetic coupling}

\shortauthors{S. Abkin and P. Daripa}

\title [mode = title]{Sign-switched antithetic coupling in gradient random-walk methods for convection--diffusion equations}

\author[1]{Stephen Abkin}[orcid=0009-0000-5981-9721]
\cormark[1]
\ead{stephen122204@tamu.edu}

\affiliation[1]{organization={Department of Mathematics, Texas A\&M University},
            addressline={3368 TAMU},
            city={College Station},
            postcode={77843-3368},
            state={TX},
            country={USA}}

\author[1]{Prabir Daripa}[orcid=0000-0002-8771-0149]
\ead{daripa@tamu.edu}

\cortext[1]{Corresponding author}

\begin{abstract}
\noindent
{\ed We propose an antithetic coupling that reduces the computational cost of gradient random-walk methods for one-dimensional convection--diffusion equations. The methods represent the solution's spatial derivative by particles with positive or negative masses, whose cumulative sums reconstruct the solution, and model diffusion by Gaussian displacements. The particles reorder and, for a nonlinear flux, interact. Antithetic sampling negates the displacements of a second simulation, which positively correlates the contributions of particles matched by position when their masses have opposite signs. Our coupling instead gives matched particles of opposite sign identical displacements and retains opposite displacements for equal-sign pairs. We prove that each simulation keeps the distribution of an uncoupled run as the matching changes, so the averaged solution keeps the bias of one simulation. For fixed particle configurations and matching, the coupling minimizes each matched pair's conditional variance contribution after diffusion among all couplings with the same Gaussian marginal distributions. From the same configurations, its variance reduction over antithetic sampling at the final diffusion stage has a closed form for any flux. For the heat equation, a martingale decomposition extends both results to each step's contribution to the terminal variance. In a Burgers' equation test, the coupling meets a tolerance for the spatially integrated mean squared error with $4$ simulations, against $34$ for antithetic sampling and $65$ for independent sampling. In a second test, it meets a common error tolerance in about $40\%$ less time than pairing particles only within each sign.}
\end{abstract}

\begin{keywords}
Gradient random-walk methods \sep Burgers' equation \sep Antithetic variates \sep Variance reduction \sep Convection--diffusion equations \sep Monte Carlo methods

\MSC[2020]{65C05, 65C35, 65M75}
\end{keywords}

\maketitle

\section{Introduction}
\label{sec:intro}

For the scalar convection--diffusion equation
\[
 \partial_t u+\partial_x f(u)=\nu\,\partial_{xx}u,\qquad x\in\R,\quad t>0,
\]
with flux $f$ and viscosity $\nu$, one class of Monte Carlo methods represents the spatial derivative $\partial_x u$ of the solution by particles that carry signed masses. Each operator-splitting step moves the particles with the velocity $f'(u)$, sorts them by position, and then, in its diffusion stage, gives each an independent Gaussian displacement. Summing the masses from the left \citep{ghoniem1985,shermanpeskin1986,roberts1989} recovers the particle approximation of the solution, which we call the \emph{field}. For a nonlinear flux, the velocity depends on the field, so the particles interact \citep{jourdain2000}. Methods of this class, referred to as gradient random-walk methods, are grid-free and self-adaptive, concentrating particles where the solution changes rapidly \citep{shermanmascagni1994,bertaglia2024}. Applications include models of nerve conduction \citep{shermanpeskin1986,shermanpeskin1988} and relaxation approximations of hyperbolic systems \citep{bertaglia2024}.

The standard error of the field averaged over $R$ independent simulations decreases like $R^{-1/2}$ \citep[Section~1.1.1]{glasserman2004}, so halving it takes four times as many simulations. Antithetic sampling, a classical way to lower this cost, averages a simulation with a second one driven by the negated random inputs, which are distributed like the original ones \citep{hammersley1956}, \citep[Section~4.2]{glasserman2004}. For the fields $U^A$ and $U^B$ of the two simulations at a point $x$, the variance of their average, the \emph{pair mean}, splits as
\[
 \Var\Bigl(\frac{U^A+U^B}{2}\Bigr)
 =\frac{\Var U^A+\Var U^B}{4}
 +\frac{\operatorname{Cov}\bigl(U^A,U^B\bigr)}{2}.
\]
Any coupling of the two simulations (a joint distribution of their random inputs) that keeps the law (distribution) of each leaves the two variances unchanged, so the pair mean is more accurate than the mean of two independent simulations exactly when the covariance is negative \citep[Section~4.2]{glasserman2004}.

Whether a coupling makes the covariance negative depends on the signs of the masses, because a rightward move lowers the field for a positive mass and raises it for a negative one. Coupling two particle simulations starts by matching each particle of one simulation with a particle of the other, which we call its \emph{partner}. Reusing the sort of each step, we pair the particles in order of position, leftmost with leftmost, and call this \emph{rank matching}. Both simulations start from the same particles, so partners initially share their sign. As the two simulations diverge, partners of opposite sign first arise when neighbors of opposite sign exchange order in one simulation alone, near peaks and troughs of the solution, where its derivative changes sign.

Two partners, whose displacements are driven by standard normal random variables $Z^A$ and $Z^B$, can be reflected, with opposite displacements ($Z^B=-Z^A$), or synchronized, with identical displacements ($Z^B=Z^A$). The two choices are the reflection (``mirror'') and synchronous couplings of two diffusions \citep[Section~3.1.1]{nusken2018}. Each step assigns its Gaussian draws to the particles in order of position, so negating the draws (as standard antithetic sampling does) reflects every rank-matched pair. We take this coupling as our antithetic baseline and call it \emph{reflected pairing}. {\ed We call the particle configurations of both simulations at the start of a diffusion stage the \emph{paired state}. For a fixed paired state and matching,} reflection minimizes, over all couplings of two partners' displacements, their contribution to the variance of the pair mean just after the stage if their signs agree, and maximizes it if they differ (Proposition~\ref{prop:pairwise}). Partners and their signs are known before the displacements are drawn, so each pair can be given the coupling that is best for its signs.

We investigate whether this sign-based coupling reduces the cost of accurate field estimates while preserving each simulation's distribution. Because particles reorder during nonlinear evolution, we examine both the variance after individual diffusion stages and the accuracy and cost of complete simulations.

\paragraph{Contributions.} We address this question by retaining rank matching and synchronizing partners of opposite sign, while partners of equal sign stay reflected. We call this antithetic coupling \emph{sign-switched pairing} (\emph{switching}). For partners with masses $m^A$ and $m^B$, its \emph{sign rule} reads
\[
 Z^B=-\sgn\bigl(m^A\bigr)\,\sgn\bigl(m^B\bigr)\,Z^A,
\]
where $\sgn(m)\in\{-1,+1\}$ is the sign of $m$. As in \citep[Proposition~56]{nusken2018}, the sign of a product decides between reflection and synchronization, but here the masses give that sign exactly, with nothing to solve or approximate. Switching changes only the diffusion stage and adds one sign comparison per matched pair, so it is easy to implement and costs about as much per simulation as reflected pairing (Sections~\ref{sec:rules} and~\ref{sec:test1}).

Although partners and their signs change as the particles reorder, each simulation keeps exactly the law of an uncoupled simulation (Proposition~\ref{prop:law}). The pair mean in turn has the bias of one simulation (Corollary~\ref{cor:bias}). {\ed For a fixed paired state and matching}, the Fr{\'e}chet--Hoeffding bounds on joint probabilities show that the sign rule minimizes each matched pair's variance contribution immediately after diffusion (Proposition~\ref{prop:pairwise}). Reflection has this property only for partners of equal sign. For a given paired state entering the final stage, the \emph{final-stage gain} is the reduction in the spatially integrated variance of the terminal pair mean when switching replaces reflected pairing. Writing integrated covariances as expected partner distances gives this gain in closed form for any flux (Proposition~\ref{prop:finalstage}). The gain is positive whenever partners of opposite sign exist and is largest when they are close. If switching is used only at the final stage, the expected gain is also the reduction in terminal variance (Corollary~\ref{cor:finalonly}). For the heat equation, a martingale decomposition splits the variance of the terminal pair mean into one contribution per stage. Given the paired state entering a stage, the sign rule under rank matching minimizes each matched pair's share of that stage's contribution, and its gain again has a closed form (Proposition~\ref{prop:everystage}).

We compare switching with an alternative that reflects pairs matched in order of position within each sign, so that all partners share a sign. We call it \emph{within-sign pairing}. Its partners can lie farther apart, because rank matching minimizes the total partner distance (Remark~\ref{rem:sorted}). The comparison tests whether keeping rank matching improves on matching within each sign.

The experiments assess the accuracy and computational cost of switching in complete nonlinear simulations. For Burgers' equation with two pulses of opposite sign, switching removes about $95\%$ of the integrated variance of the mean of two independent simulations and reduces that of within-sign pairing by $26\%$ (Example~1, Section~\ref{sec:test1}). In Example~2 (Section~\ref{sec:work}), switching meets a common {\ed tolerance} for the integrated mean squared error with $4$ simulations, where reflected pairing uses $34$ simulations and independent sampling uses $65$ simulations. Switching takes about an eighth of the execution time of reflected pairing, and at a second, looser {\ed tolerance} with fewer particles, it takes about $40\%$ less time than within-sign pairing. Example~3 (Section~\ref{sec:structure}) finds the variance reduction over reflected pairing largest when particles of opposite sign are close, with nearly equal ratios for Burgers' equation and the heat equation. Examples~4--6 (Sections~\ref{sec:profiles}--\ref{sec:observables}) vary the initial profile, final time, viscosity, and flux, combine switching with control variates and with conditional Monte Carlo over the final diffusion stage, and apply it to nonlinear quantities computed from the solution.

\paragraph{Related work.} N{\"u}sken and Pavliotis~\citep{nusken2018} couple copies of a diffusion process for Monte Carlo sampling by a sign-based choice between reflection and synchronous coupling. The sign is that of a product of derivatives of the solution of a Poisson equation \citep[Proposition~56]{nusken2018}. They also propose choosing the sign from derivatives of the function being averaged \citep[Eq.~(148)]{nusken2018}. Consecutive copies are coupled after sorting them by the magnitude of the Poisson solution's gradient \citep[Section~5.2]{nusken2018}. We couple two simulations whose particles reorder as they move.

L{\'e}cot~\citep{lecot_rims2001} and L{\'e}cot and Ogawa~\citep{lecot2002} sort particles by position to organize quasi-random diffusion inputs within one simulation of Burgers' equation. Morokoff and Caflisch~\citep{morokoff1993} do the same for the heat equation. Gradient random-walk methods have convergence and error analyses for nonlinear convection--diffusion equations \citep{roberts1989,bossytalay1997,bossy2004,bossyjourdain2002,jourdain2000,bencheikh2020}. Convergence is also proved for a related particle Monte Carlo method for relaxation approximations of scalar conservation laws \citep{towers2025} and for a stochastic interacting particle-field method for the three-dimensional Keller--Segel chemotaxis system \citep{huwangxinzhang2026}.

Quasi-Monte Carlo methods reduce the variance of Markov-chain and interacting-particle simulations \citep{lecuyer2008,benrached2024,netterdon2026}. Coupling appears in Gibbs sampling \citep{frigessi2000}, parallel antithetic Markov chain Monte Carlo \citep{craiumeng2005}, particle filtering \citep{sen2018,jacob2020}, and multilevel Monte Carlo \citep{giles2008}. In multilevel Monte Carlo, Haji-Ali and Tempone \citep[Section~3.2, Eq.~(14)]{hajiali2018} apply a partitioning estimator to the difference of two levels. Bencheikh and Jourdain study it numerically as ``antithetic sampling'' \citep[Section~2, Eq.~(10)]{bencheikhjourdain2019}. The antithetic estimator of Giles and Szpruch \citep[Section~2.2]{gilesszpruch2014} and the estimator of Szpruch and Tse for interacting particle systems \citep{szpruchtse2021} also act on level differences. We couple simulations of equal particle count and time step.

\paragraph{Scope.} We work with one-dimensional scalar equations and fixed particle masses, where sorting by position supplies rank matching. The heat equation admits a stage-by-stage analysis because its terminal field responds monotonically to each displacement, a property that reordering under Burgers' flux can reverse (Remark~\ref{rem:reversal}). Burgers' equation and a nonconvex cubic flux supply the nonlinear tests. For monotone data, whose masses all have one sign, switching coincides with reflected pairing (Lemma~\ref{lem:conditional}).

\paragraph{Outline.} Section~\ref{sec:background} reviews the solver, Section~\ref{sec:coupling} derives sign-switched pairing and proves that it keeps the law of each simulation, Section~\ref{sec:analysis} analyzes its gain over reflected pairing, Section~\ref{sec:numerics} reports six numerical examples, and Section~\ref{sec:conclusions} concludes.

\section{Gradient random-walk solver}
\label{sec:background}

We recall the gradient random-walk solver and its estimators. The solver combines the signed cumulative reconstruction of L{\'e}cot \citep[Eqs.~(35)--(37)]{lecot_rims2001} and Jourdain \citep{jourdain2000} with operator splitting of transport and diffusion \citep{chorin1973,holden2010}. The representation follows our earlier study \citep{abkindaripa2026}, with Burgers' equation now solved by direct nonlinear transport instead of the Cole--Hopf transformation \citep{hopf1950,cole1951}.

\paragraph{Notation.} We write $\Pr$, $\E$, $\Var$, and $\operatorname{Cov}$ for probability, expectation, variance, and covariance, and $\ind\{E\}$ for the indicator of an event $E$. The standard normal distribution function and density are $\Phi$ and $\phi$, and $\mathcal N(0,I_N)$ is the law of a vector of $N$ independent standard normal random variables. The sign of a real number $m\ne0$ is $\sgn(m)\in\{-1,+1\}$, and $\delta_a$ is the unit point mass at $a\in\R$. Subscripts $i$ and $j$ index particles by rank (their place in left-to-right order), and $k$ indexes the time step and its diffusion stage (Section~\ref{sec:onestep}).

\subsection{From the solution to its gradient}
\label{sec:representation}

We consider the scalar convection--diffusion equation
\begin{equation}
 \partial_t u+\partial_x f(u)=\nu\,\partial_{xx}u,\qquad
 u(x,0)=u_0(x),\quad x\in\R,\quad 0<t\le T,
 \label{eq:pde}
\end{equation}
where $f\in C^2(\R)$ is the flux, $\nu>0$ the viscosity, and $T$ the final time. We assume that $u_0$ is locally absolutely continuous with $u_0'\in L^1(\R)$, so that $u_0$ has bounded variation and finite far-field values $u_\pm=\lim_{x\to\pm\infty}u_0(x)$. The experiments use Burgers' flux $f(u)=u^2/2$, the cubic flux $f(u)=u^3/3$, and $f\equiv0$, which gives the heat equation.

The solver evolves the derivative $w=\partial_x u$. Differentiating \eqref{eq:pde} in $x$ gives
\begin{equation}
 \partial_t w+\partial_x\bigl(f'(u)\,w\bigr)=\nu\,\partial_{xx}w,
 \qquad w(x,0)=u_0'(x).
 \label{eq:gradient-pde}
\end{equation}
Equation~\eqref{eq:gradient-pde} transports $w$ with the velocity $f'(u)$ and diffuses it \citep{jourdain2000,jourdain2006}, and its conservation form preserves the total mass $\int_\R w\,\mathrm dx=u_+-u_-$. The velocity depends on $w$ through $u=u_-+\int_{-\infty}^x w\,\mathrm dy$, so for nonlinear $f$ the equation is nonlinear and nonlocal.

Gradient random-walk methods approximate $w$ by point masses that move with the velocity $f'(u)$ and receive random displacements \citep{ghoniem1985,roberts1989,shermanmascagni1994}. We use $N$ particles with positions $X_i(t)$ and fixed signed masses $m_i$. Each mass stays with its particle, and sorting relabels the particles by their current rank (Section~\ref{sec:onestep}). We recover $u$ by summing the masses from the left,
\begin{equation}
 w^N(\cdot,t)=\sum_{i=1}^N m_i\,\delta_{X_i(t)},
 \qquad
 U(x,t)=u_-+\!\!\sum_{i:\,X_i(t)\le x}\!\! m_i .
 \label{eq:representation}
\end{equation}
The field $U$ is a step function whose jump at each particle position is the sum of the masses located there (we leave its dependence on $N$ implicit). Its right far-field value is $u_-+\sum_i m_i$, so the masses satisfy
\begin{equation}
 \sum_{i=1}^N m_i=u_+-u_-.
 \label{eq:total-mass}
\end{equation}

If particle $i$ moves from $a$ to $a'>a$ while the others stay fixed, the field changes only on the interval it sweeps out,
\[
 U_{\mathrm{new}}(x)-U_{\mathrm{old}}(x)
 =m_i\bigl(\ind\{x\ge a'\}-\ind\{x\ge a\}\bigr)
 =-m_i\ind\{a\le x<a'\}.
\]
A rightward move therefore lowers the field for a positive mass and raises it for a negative one, so the coupling of Section~\ref{sec:coupling} uses the masses as well as the positions.

All simulations, including both simulations of a pair, start from the same deterministic state. With $u_+=u_-$, as for all our initial profiles, and $N$ even, we use $N/2$ particles of each sign with mass magnitude $\mathrm{TV}(u_0)/N$, so that \eqref{eq:total-mass} holds exactly. Here $\mathrm{TV}(u_0)=\int_\R|u_0'(x)|\,\mathrm dx$ is the total variation of $u_0$. The particles of each sign start at the quantiles of orders $(j-\tfrac12)/(N/2)$, $j=1,\ldots,N/2$, of the probability density proportional to $\max(\pm u_0',0)$, the ``optimal deterministic initialization'' of Bencheikh and Jourdain \citep[Section~2.2]{bencheikh2020} applied to each sign. Example~3 instead uses fixed particle configurations (Section~\ref{sec:structure}).

\subsection{One step of a simulation}
\label{sec:onestep}

We advance the particles over $K$ steps of size $h=T/K$, alternating transport and diffusion. Step $k$, $k=0,\ldots,K-1$, takes them from time $kh$ to $(k+1)h$. Its diffusion part, the $k$th \emph{diffusion stage} (stage~$k$ for short), gives every particle an independent Gaussian displacement with standard deviation
\begin{equation}
 \sigma=\sqrt{2\nu h},
 \label{eq:sigma}
\end{equation}
which represents the diffusion term of \eqref{eq:gradient-pde} by Brownian motion \citep{ghoniem1985}. The coupling of Section~\ref{sec:coupling} acts only on these stages. Step~$k$ applies, in order, four updates to the positions $X_i$, masses $m_i$, and velocities $v_i$, $i=1,\dots,N$.
\begin{enumerate}
\item \emph{Transport.} Each particle moves with its velocity,
\[
 X_i\leftarrow X_i+v_i\,h .
\]
\item \emph{Sort.} A stable sort reorders the positions and masses by increasing position, keeping tied particles in their current order, and so sets the ranks for the velocity update and rank matching.
\item \emph{Velocity.} Each particle receives the velocity
\begin{equation}
 v_i\leftarrow f'\Bigl(u_-+\sum_{j\le i}m_j\Bigr).
 \label{eq:velocity}
\end{equation}
\item \emph{Diffusion stage.} Each particle receives its Gaussian displacement,
\[
 X_i\leftarrow X_i+\sigma Z_{k,i},
\]
where the vector $Z_k\sim\mathcal N(0,I_N)$ is drawn independently of $Z_0,\dots,Z_{k-1}$.
\end{enumerate}
Each particle keeps its velocity through the diffusion stage, so the transport of step~$k+1$ uses the velocities computed in step~$k$. Before the first step the initial positions are sorted and the velocities computed by \eqref{eq:velocity}.

For distinct positions, the velocity \eqref{eq:velocity} is the characteristic speed $f'$ \citep[Section~11.2]{leveque2002} at the value of $U$ just to the right of the jump of particle $i$, and at tied positions the sum follows the rank order of the stable sort. It is the drift of Jourdain's interacting particle system with signed weights \citep{jourdain2000}, and Bencheikh and Jourdain \citep[Eqs.~(1.6)--(1.7)]{bencheikh2020} analyze a closely related form for nondecreasing data.

\subsection{Estimators and error}
\label{sec:estimators}

We compare two estimators, both sample means of $R$ independent and identically distributed realizations. In \emph{independent sampling} a realization is one simulation, and the estimator averages the fields $U_1,\dots,U_R$ at time $T$,
\begin{equation}
 \hat u^{\mathrm{ind}}_R(x)=\frac1R\sum_{r=1}^{R}U_r(x).
 \label{eq:est-single}
\end{equation}

In \emph{paired sampling} a realization is a pair of simulations $A$ and $B$ advanced with coupled random inputs (Section~\ref{sec:coupling}), so the fields $U^A_r$ and $U^B_r$ of the $r$th pair are dependent, while different pairs are independent. We take the \emph{pair mean} and average it over the pairs,
\begin{equation}
 \bar U_r(x)=\tfrac12\bigl(U^A_r(x)+U^B_r(x)\bigr),
 \label{eq:est-pair}
\end{equation}
\begin{equation}
 \hat u^{\mathrm{pair}}_R(x)=\frac1R\sum_{r=1}^{R}\bar U_r(x),
 \label{eq:est-ens}
\end{equation}
which is the classical antithetic estimator \citep[Section~4.2]{glasserman2004} applied to complete simulations and uses $2R$ simulations. Without the index $r$, $U$ denotes the field of one uncoupled simulation, and $U^A$, $U^B$, and $\bar U$ denote the fields of one pair.

We measure accuracy on the interval $\Omega=[-5,5]$ at $L=400$ equispaced points $x_1,\dots,x_L$ with spacing $\Delta x=10/399$, in the discrete norm
\[
 \|e\|_\Omega^2=\Delta x\sum_{\ell=1}^{L}e(x_\ell)^2 .
\]
The error of an estimator $\hat u_R$ is its integrated mean squared error (integrated MSE)
\begin{equation}
 \mathcal E(R)=\E\bigl\|\hat u_R-u_{\mathrm{ref}}\bigr\|_\Omega^2 ,
 \label{eq:error}
\end{equation}
where $u_{\mathrm{ref}}$ is a deterministic reference solution (Section~\ref{sec:setup}).

For paired sampling $\E\hat u_R=\E\bar U$ and $\Var\hat u_R(x)=\Var\bar U(x)/R$. The standard bias--variance decomposition splits \eqref{eq:error} into the squared bias $b^2$ and the integrated variance $V$ of one realization,
\begin{equation}
 \mathcal E(R)=b^2+\frac{V}{R},
 \qquad
 b^2=\bigl\|\E\bar U-u_{\mathrm{ref}}\bigr\|_\Omega^2,
 \qquad
 V=\Delta x\sum_{\ell=1}^{L}\Var\bigl(\bar U(x_\ell)\bigr).
 \label{eq:decomposition}
\end{equation}
The same formulas hold for independent sampling with $\bar U$ replaced by $U$.

The squared bias contains the finite-particle and time-stepping errors, and more realizations decrease $V/R$ while leaving $b^2$ unchanged. We require every coupling to preserve the law of each simulation. Such a coupling keeps the $b^2$ of independent sampling (Corollary~\ref{cor:bias}) and changes $\mathcal E(R)$ only through $V$, the quantity that Section~\ref{sec:coupling} targets. For a particular coupling $\bullet$ we write $V_\bullet$ and $\mathcal E_\bullet(R)$.

\section{Sign-switched pairing}
\label{sec:coupling}

We derive the sign rule from the variance of one matched pair at one diffusion stage, apply it at every stage, and show that each simulation keeps its uncoupled law.

\subsection{Variance of one matched pair}
\label{sec:motivation}

At a diffusion stage, we match the particles of $A$ one to one with those of $B$, for instance by rank matching, which pairs equal ranks (Section~\ref{sec:intro}). Let particle $i$ of $A$ and its partner in $B$ have masses $m^A_i$ and $m^B_i$ and random displacements $\sigma Z^A_i$ and $\sigma Z^B_i$. We call the choice $Z^B_i=-Z^A_i$ \emph{reflection} of the pair and $Z^B_i=Z^A_i$ its \emph{synchronization}.

By \eqref{eq:representation} the field after the stage falls with the displacement of a positive mass and rises with that of a negative one. The classical antithetic argument \citep{rubinstein1985}, \citep[Section~4.2]{glasserman2004} states that two monotone functions of one random variable are negatively correlated when one increases and the other decreases, and positively correlated when both move the same way. Reflection therefore correlates the contributions of partners to the pair mean negatively when their signs agree and positively when they differ, and synchronization does the opposite. The sign rule keeps the negative correlation in both cases by setting $Z^B_i=\varepsilon_iZ^A_i$ with the multiplier
\[
 \varepsilon_i=-\sgn\bigl(m^A_i\bigr)\sgn\bigl(m^B_i\bigr).
\]
As in \citep[Proposition~56]{nusken2018}, the sign of a product decides between reflection and synchronization.

To compare the sign rule with every other coupling, in this subsection we fix the matching, condition on the states of both simulations entering the stage, and write $\bar U(x)$ for the pair mean just after it. The partners of pair $i$, $i=1,\dots,N$, enter the stage at positions $a_i$ and $b_i$ and leave it at
\[
 X^A_i=a_i+\sigma Z^A_i,\qquad X^B_i=b_i+\sigma Z^B_i .
\]
A coupling of pair $i$ is a joint law of $(Z^A_i,Z^B_i)$ with standard normal marginal distributions. The displacements of different matched pairs are independent, so each simulation alone receives the independent Gaussian displacements of the solver (Section~\ref{sec:law}).

Every particle belongs to exactly one matched pair, so \eqref{eq:representation} gives
\[
 \bar U(x)=u_-+\tfrac12\sum_{i=1}^N\bigl(m^A_i\chi^A_i+m^B_i\chi^B_i\bigr),
\]
where $\chi^A_i=\ind\{X^A_i\le x\}$ and $\chi^B_i=\ind\{X^B_i\le x\}$. Since distinct matched pairs are independent, the variance of $\bar U(x)$ is the sum over these pairs of
\begin{equation}
 \tfrac14\Bigl[(m^A_i)^2\Var\chi^A_i+(m^B_i)^2\Var\chi^B_i
  +2\,m^A_im^B_i\operatorname{Cov}\bigl(\chi^A_i,\chi^B_i\bigr)\Bigr].
 \label{eq:pairvar}
\end{equation}

The indicators are Bernoulli variables with success probabilities
\[
 p_i=\Pr\{X^A_i\le x\}=\Phi\!\Bigl(\frac{x-a_i}{\sigma}\Bigr),
 \qquad
 q_i=\Pr\{X^B_i\le x\}=\Phi\!\Bigl(\frac{x-b_i}{\sigma}\Bigr),
\]
so $\Var\chi^A_i=p_i(1-p_i)$ and $\Var\chi^B_i=q_i(1-q_i)$. By the Fr{\'e}chet--Hoeffding bounds \citep[see][Theorem~2.2.3 and Section~2.5]{nelsen2006}, for every coupling the joint probability satisfies
\begin{equation}
 \max(p_i+q_i-1,\,0)\;\le\;\Pr\{X^A_i\le x,\ X^B_i\le x\}\;\le\;\min(p_i,q_i).
 \label{eq:frechet}
\end{equation}
The covariance satisfies
\[
 \operatorname{Cov}\bigl(\chi^A_i,\chi^B_i\bigr)=\Pr\{X^A_i\le x,\ X^B_i\le x\}-p_iq_i,
\]
so for every coupling it lies between $\gamma^-_i(x)$ and $\gamma^+_i(x)$, where
\begin{equation}
 \begin{aligned}
 \gamma^-_i(x)&=\max(p_i+q_i-1,0)-p_iq_i\;\le\;0,\\
 \gamma^+_i(x)&=\min(p_i,q_i)-p_iq_i\;\ge\;0.
 \end{aligned}
 \label{eq:cpm}
\end{equation}

The covariance enters \eqref{eq:pairvar} with weight $2m^A_im^B_i$, so a positive mass product calls for the lower bound in \eqref{eq:cpm} and a negative one for the upper bound.

\begin{proposition}[optimality of the sign rule for one pair]
\label{prop:pairwise}
Fix a diffusion stage, a matching, and the states of both simulations before the stage. For a matched pair $i$ with nonzero masses, consider all joint laws of $(Z^A_i,Z^B_i)$ with standard normal marginal distributions, the displacements of distinct matched pairs being independent. The contribution \eqref{eq:pairvar} of pair $i$ to the conditional variance of $\bar U(x)$ after the stage is minimized, simultaneously for every $x\in\R$, by
\[
 Z^B_i=\varepsilon_i\,Z^A_i,\qquad \varepsilon_i=-\sgn\bigl(m^A_i\bigr)\sgn\bigl(m^B_i\bigr).
\]
The minimum equals $\tfrac14\bigl[(m^A_i)^2p_i(1-p_i)+(m^B_i)^2q_i(1-q_i)\bigr]+\tfrac12m^A_im^B_i\,\gamma_i(x)$, with $\gamma_i=\gamma^-_i$ when $m^A_im^B_i>0$ and $\gamma_i=\gamma^+_i$ when $m^A_im^B_i<0$. Reflecting a pair with $m^A_im^B_i<0$ attains the pointwise \emph{maximum} instead.
\end{proposition}

\begin{proof}
The variance terms of \eqref{eq:pairvar} do not depend on the coupling, and by \eqref{eq:frechet} the covariance lies in $[\gamma^-_i(x),\gamma^+_i(x)]$. Under reflection and under synchronization, both particles end at or to the left of $x$ exactly on the events
\[
 \{X^A_i\le x,\ X^B_i\le x\}=
 \begin{cases}
 \bigl\{-\tfrac{x-b_i}{\sigma}\le Z^A_i\le\tfrac{x-a_i}{\sigma}\bigr\}, & Z^B_i=-Z^A_i,\\[4pt]
 \bigl\{Z^A_i\le\min\bigl(\tfrac{x-a_i}{\sigma},\tfrac{x-b_i}{\sigma}\bigr)\bigr\}, & Z^B_i=Z^A_i,
 \end{cases}
\]
which have probabilities $\max(p_i+q_i-1,0)$ and $\min(p_i,q_i)$ because $\Phi(-y)=1-\Phi(y)$. The two couplings thus attain the two endpoints at every $x$, as in Whitt's construction of the extremal joint distributions \citep[Theorem~2.5]{whitt1976}. By the weight $2m^A_im^B_i$ in \eqref{eq:pairvar}, reflection therefore minimizes the contribution when $m^A_im^B_i>0$, and when $m^A_im^B_i<0$ synchronization minimizes it and reflection maximizes it.
\end{proof}

Given the paired state entering the stage and the matching, reflection thus minimizes a matched pair's contribution to the conditional variance of $\bar U(x)$ just after the stage if the partners have equal signs, and maximizes it if their signs differ. Section~\ref{sec:analysis} evaluates the gain of switching over reflected pairing, and Section~\ref{sec:matching} analyzes the matching, which sets the positions and masses in \eqref{eq:pairvar}.

\subsection{Coupling two simulations}
\label{sec:rules}

To couple complete simulations, for $k=0,\dots,K$ let
\[
 \Fil_k=\sigma\bigl(Z_0,\dots,Z_{k-1}\bigr)
\]
be the $\sigma$-algebra generated by the Gaussian vectors $Z_0,\dots,Z_{k-1}\in\R^N$ drawn at the first $k$ stages. Since both simulations start from the same deterministic state, $\Fil_k$ determines their states entering stage $k$. An $\Fil_k$-measurable rule uses only what is known before $Z_k$ is drawn, the property behind law preservation (Section~\ref{sec:law}).

\begin{definition}[admissible coupling]
\label{def:coupling}
An \emph{admissible coupling} chooses, at each stage $k$, a bijection $\pi_k$ of the ranks $\{1,\dots,N\}$ and multipliers $\varepsilon_{k,1},\dots,\varepsilon_{k,N}\in\{-1,+1\}$, collected in the vector $\varepsilon_k$, both $\Fil_k$-measurable. It then draws $Z_k\sim\mathcal N(0,I_N)$ independent of $\Fil_k$ and sets
\begin{equation}
  X^A_i \leftarrow X^A_i+\sigma Z_{k,i},
  \qquad
  X^B_j \leftarrow X^B_j+\sigma\,\varepsilon_{k,j}\,Z_{k,\pi_k(j)} .
  \label{eq:coupling}
\end{equation}
Thus particle $j$ of $B$ is matched with particle $\pi_k(j)$ of $A$, and it receives the same displacement as its partner when $\varepsilon_{k,j}=1$ and the opposite displacement when $\varepsilon_{k,j}=-1$.
\end{definition}

We compare three admissible couplings, which differ only in $(\pi_k,\varepsilon_k)$.
\begin{itemize}
\item \emph{Reflected pairing} ($\refl$), our baseline, reflects every rank-matched pair, so $\pi_k=\mathrm{id}$ and $\varepsilon_{k,j}=-1$.
\item \emph{Sign-switched pairing} ($\swi$) keeps rank matching and applies the sign rule at every stage, with $\pi_k=\mathrm{id}$ and
  \begin{equation}
    \varepsilon_{k,j}=-\,\sgn\!\bigl(m^A_j\bigr)\,\sgn\!\bigl(m^B_j\bigr),
    \label{eq:switch}
  \end{equation}
  where $m^A_j$ and $m^B_j$ are the nonzero masses of the particles at rank $j$ entering stage $k$.
\item \emph{Within-sign pairing} ($\wit$) reflects every pair formed by \emph{within-sign matching}, which matches the $n$th positive particle of $B$ from the left with the $n$th positive particle of $A$, and likewise for the negative particles. Thus $\varepsilon_{k,j}=-1$ and $\pi_k$ is this matching. It requires nonzero masses and equal sign counts in $A$ and $B$, which a common initial state and fixed masses ensure. Its partners can lie several particles of the other sign apart (Remark~\ref{rem:sorted}).
\end{itemize}

Each coupling chooses $\pi_k$ and $\varepsilon_k$ after the transport, sort, and velocity updates of step~$k$ (Section~\ref{sec:onestep}). The sign rule adds one comparison per matched pair, $O(N)$ operations per stage, to the solver's $O(N\log N)$ sort.

\subsection{Law preservation}
\label{sec:law}

Viewed alone, each simulation of a pair has the law of an uncoupled simulation, because the matching and multipliers of each stage are fixed before the draw.

\begin{proposition}[law preservation]
\label{prop:law}
Let two simulations be advanced from the same initial state by the steps of Section~\ref{sec:onestep}, with their diffusion draws coupled by an admissible coupling (Definition~\ref{def:coupling}). For $k=0,\ldots,K$, let $\Xi^A_k$ and $\Xi^B_k$ be the complete states $(X,m,v)$ of simulations $A$ and $B$ after $k$ steps. Then the law of the trajectory $(\Xi^A_k)_{k=0}^K$ is exactly the law of an uncoupled simulation, and so is the law of $(\Xi^B_k)_{k=0}^K$. Only the joint law of the two trajectories differs from that of two independent simulations.
\end{proposition}

\begin{proof}
The transport, sort, and velocity updates of each step act deterministically on the state, so it suffices to compare the displacements of the stages. At stage $k$ simulation $A$ receives $\sigma Z_k$, and by \eqref{eq:coupling} simulation $B$ receives $\sigma\widetilde Z_k$ with
\[
 \widetilde Z_k=D_kP_kZ_k,\qquad (P_kz)_j=z_{\pi_k(j)},\qquad D_k=\operatorname{diag}(\varepsilon_k).
\]
Since $\pi_k$ is a bijection, $\varepsilon_{k,j}=\pm1$, and both are $\Fil_k$-measurable, the matrix $D_kP_k$ is orthogonal and fixed before the draw, as in \citep[Lemma~74]{nusken2018}. Conditionally on $\Fil_k$ the vector $Z_k$ is still $\mathcal N(0,I_N)$, a law invariant under orthogonal maps, so the conditional law of $\widetilde Z_k$ given $\Fil_k$ is $\mathcal N(0,I_N)$. This law does not depend on $\Fil_k$, so $\widetilde Z_k$ is independent of $\Fil_k$ and hence of the $\Fil_k$-measurable vectors $\widetilde Z_0,\dots,\widetilde Z_{k-1}$. By induction on $k$, the vectors $\widetilde Z_0,\dots,\widetilde Z_{K-1}$ are independent with law $\mathcal N(0,I_N)$. Simulation $B$ is therefore the image of this sequence under the deterministic map of an uncoupled simulation (transport, sort with the same rule for ties, velocity, and diffusion), so its trajectory has the uncoupled law. Simulation $A$ is the case $\widetilde Z_k=Z_k$, and the joint law differs because $\widetilde Z_k$ depends on $Z_k$.
\end{proof}

Law preservation fixes the expected field of every simulation, and hence the bias of the pair mean.

\begin{corollary}[common bias]
\label{cor:bias}
Under any admissible coupling, $\E\bar U(x)=\E U(x)$ for every $x$, where $U$ is the field of one uncoupled simulation. Hence the squared bias $b^2$ in \eqref{eq:decomposition} takes the same value for reflected, sign-switched, and within-sign pairing, and equals that of independent sampling.
\end{corollary}

\begin{proof}
The expectation of $U$ exists because $|U|\le|u_-|+\sum_i|m_i|$. By Proposition~\ref{prop:law}, $U^A$ and $U^B$ each have the law of $U$, and linearity of expectation gives
\[
 \E\bar U(x)=\tfrac12\bigl(\E U^A(x)+\E U^B(x)\bigr)=\E U(x).
 \qedhere
\]
\end{proof}

\begin{remark}[nonlinear functions of the pair mean]
\label{rem:nonlinear}
The argument of Corollary~\ref{cor:bias} applies to any integrable quantity evaluated on each simulation and then averaged. A nonlinear function of the pair mean is a different estimator, whose bias can depend on the coupling, and Example~6 compares the two (Section~\ref{sec:observables}).
\end{remark}

\section{Variance analysis}
\label{sec:analysis}
\label{sec:gain}

Switching \eqref{eq:switch} departs from reflected pairing only for partners of opposite sign, which it synchronizes so that the diffusion stage preserves their distance. Reflection increases their expected distance, and with it their contribution to the conditional variance just after the stage, integrated over $\R$. We show that switching coincides with reflected pairing until such partners appear and, for any flux, how much switching at the final stage reduces the conditional variance of the terminal pair mean given the paired state entering that stage. For the heat equation, we compute the corresponding reduction of each stage's contribution in a martingale decomposition of the terminal variance. We then relate the gain to the matching.

\subsection{When switching and reflected pairing coincide}

Partners of opposite sign first appear when neighbors of opposite sign exchange order in one simulation alone.

\begin{lemma}[coincidence with reflected pairing]
\label{lem:conditional}\label{cor:onesign}
Start reflected and sign-switched pairing from the same paired state and drive them with the same sequence of draws. Let $k^*$ be the first stage at which some matched pair has masses of opposite sign. The two paired trajectories coincide up to the state entering stage $k^*$, and if no such stage occurs among the $K$ stages, their complete trajectories coincide. In particular, if all masses share one sign, as for monotone initial data with $u_+\ne u_-$, whose particles are placed according to $|u_0'|$ and all carry the sign of $u_+-u_-$, sign-switched and reflected pairing are the same algorithm for every realization.
\end{lemma}

\begin{proof}
Both couplings match by rank, and \eqref{eq:switch} sets $\varepsilon_{k,j}=-1$, the multiplier of reflected pairing, for every pair of equal signs. Before stage $k^*$ both give the same displacements to the same states, and the transport, sort, and velocity updates are deterministic, so induction on $k$ gives both claims. Masses are fixed, so if all share one sign no such stage occurs.
\end{proof}

\subsection{Variance reduction at the final stage}
\label{sec:finalstage}

The \emph{gain} of the final stage, after which no transport follows, is the reduction in the conditional variance of the terminal pair mean, integrated over $\R$, when that stage alone uses sign-switched instead of reflected pairing,
\[
 \Gamma=\int_{\R}\Bigl[\Var_{\refl}\bigl(\bar U(x)\mid\Fil_{K-1}\bigr)
 -\Var_{\swi}\bigl(\bar U(x)\mid\Fil_{K-1}\bigr)\Bigr]\mathrm dx.
\]
The subscript names the final rule, and conditioning on $\Fil_{K-1}$ fixes the paired state entering that stage. Proposition~\ref{prop:finalstage} expresses $\Gamma$ in closed form through partner separations by two classical facts. The first, the identity behind the energy distance of Sz{\'e}kely and Rizzo \citep[Section~2]{szekelyrizzo2013}, turns the integrated covariances of Proposition~\ref{prop:pairwise} into expected partner distances.

\begin{lemma}[integrated covariance and expected distance]
\label{lem:wholeline}
For integrable $X$ and $Y$ and a random variable $Y'$ with the law of $Y$, independent of $X$,
\begin{equation}
 \int_\R\operatorname{Cov}\bigl(\ind\{X\le x\},\ind\{Y\le x\}\bigr)\,\mathrm dx
 =\tfrac12\bigl(\E|X-Y'|-\E|X-Y|\bigr).
 \label{eq:covdist}
\end{equation}
\end{lemma}

\begin{proof}
For $a\in\R$ let $\eta_a(x)=\ind\{x\ge a\}-\ind\{x\ge 0\}$, supported between $a$ and $0$. We show that for all $a,b\in\R$,
\begin{equation}
  \int_{\R}\eta_a(x)\,\eta_b(x)\,\mathrm dx
  \;=\;\tfrac12\bigl(|a|+|b|-|a-b|\bigr).
  \label{eq:kernel}
\end{equation}
By symmetry we may assume $|a|\le|b|$. If $0\le a\le b$, then $\eta_a=\eta_b=-1$ on $[0,a)$, $\eta_a=0$ and $\eta_b=-1$ on $[a,b)$, and both are zero elsewhere. Hence $\int\eta_a\eta_b=|a|$, which equals $\tfrac12(|a|+|b|-|a-b|)$ because $|b|=|a|+|b-a|$. The case $0\ge a\ge b$ is identical. For $a$ and $b$ on opposite sides of $0$ the product is zero almost everywhere and $|a-b|=|a|+|b|$ makes the right-hand side zero.

For $Y'$ as in the statement,
\[
 \E[\eta_X(x)\eta_{Y'}(x)]=\E\eta_X(x)\,\E\eta_Y(x).
\]
Since $\ind\{X\le x\}=\eta_X(x)+\ind\{x\ge0\}$ and $\ind\{x\ge0\}$ is deterministic,
\[
 \operatorname{Cov}\bigl(\ind\{X\le x\},\ind\{Y\le x\}\bigr)
 =\E[\eta_X(x)\eta_Y(x)]-\E[\eta_X(x)\eta_{Y'}(x)].
\]
Because $\int_\R|\eta_a\eta_b|\le\min(|a|,|b|)$ and $X$ and $Y$ are integrable, Fubini's theorem and \eqref{eq:kernel} give
\[
 \begin{aligned}
 \int_\R\operatorname{Cov}\bigl(\ind\{X\le x\},\ind\{Y\le x\}\bigr)\,\mathrm dx
 &=\tfrac12\bigl(\E|X|+\E|Y|-\E|X-Y|\bigr)\\
 &\quad-\tfrac12\bigl(\E|X|+\E|Y'|-\E|X-Y'|\bigr).
 \end{aligned}
\]
Since $\E|Y|=\E|Y'|$, this is \eqref{eq:covdist}.
\end{proof}

The second fact is the expected absolute value of a Gaussian random variable, the mean of a folded normal distribution \citep[Eq.~(3)]{leone1961}. For $d\in\R$ and $\tau\ge0$,
\begin{equation}
 g_\tau(d)=\E\,|d+\tau Z|=2\tau\phi(d/\tau)+d\bigl(2\Phi(d/\tau)-1\bigr),
 \qquad Z\sim\mathcal N(0,1),
 \label{eq:folded}
\end{equation}
where the closed form holds for $\tau>0$.

\begin{proposition}[final-stage gain]
\label{prop:finalstage}
Fix the paired state entering the final diffusion stage, with all masses nonzero, and match particles by rank. For each rank $i$, let $a_i$ and $b_i$ be the positions of the particles of rank $i$ in $A$ and $B$, $m^A_i$ and $m^B_i$ their masses, and $d_i=a_i-b_i$ their separation. Then
\begin{equation}
  \Gamma
  \;=\;\frac14\!\!\sum_{i:\,m^A_i m^B_i<0}\!\!
    \bigl|m^A_i m^B_i\bigr|\,
    \Bigl[g_{2\sigma}(d_i)-|d_i|\Bigr]\;\ge\;0 ,
  \label{eq:finalstage}
\end{equation}
and $\Gamma>0$ whenever some matched pair has masses of opposite sign.
\end{proposition}

\begin{proof}
Given the paired state entering the final stage, both rules match by rank and displacements of different ranks are independent, so the conditional variance of the terminal pair mean is the sum over ranks of \eqref{eq:pairvar}. Its variance terms depend only on the marginal distributions and are integrable in $x$ for any total mass \eqref{eq:total-mass}, since $\Var\chi^A_i=p_i(1-p_i)$ has Gaussian tails. They cancel in the difference, which leaves
\[
 \Gamma=\tfrac12\sum_i m^A_im^B_i\int_\R\Bigl[\operatorname{Cov}_{\refl}\bigl(\chi^A_i,\chi^B_i\bigr)
 -\operatorname{Cov}_{\swi}\bigl(\chi^A_i,\chi^B_i\bigr)\Bigr]\mathrm dx .
\]
Pairs of equal sign are reflected by both rules and drop out. For a pair of opposite sign, \eqref{eq:covdist} gives under each rule
\[
 \int_\R\operatorname{Cov}\bigl(\chi^A_i,\chi^B_i\bigr)\,\mathrm dx
 =\tfrac12\bigl(\E|X^A_i-X'^B_i|-\E|X^A_i-X^B_i|\bigr),
\]
where $X'^B_i$ is a copy of $X^B_i$ independent of $X^A_i$. The term $\E|X^A_i-X'^B_i|$ depends only on the marginal distributions, so the rules differ only through the expected partner distance $\E|X^A_i-X^B_i|$ after the stage. Reflection gives $X^A_i-X^B_i=d_i+2\sigma Z_{K-1,i}$, and synchronization cancels the displacements in this difference, so
\[
 \E|X^A_i-X^B_i|=
 \begin{cases}
  g_{2\sigma}(d_i), & \text{under reflection},\\
  |d_i|, & \text{under synchronization}.
 \end{cases}
\]
Hence
\[
 \begin{aligned}
 \Gamma&=\tfrac12\!\!\sum_{i:\,m^A_im^B_i<0}\!\! m^A_im^B_i\cdot
 \tfrac12\bigl(|d_i|-g_{2\sigma}(d_i)\bigr)\\
 &=\tfrac14\!\!\sum_{i:\,m^A_im^B_i<0}\!\!\bigl|m^A_im^B_i\bigr|
 \bigl(g_{2\sigma}(d_i)-|d_i|\bigr),
 \end{aligned}
\]
which is \eqref{eq:finalstage}. Non-negativity follows from Jensen's inequality, $g_\tau(d)=\E|d+\tau Z|\ge|\E(d+\tau Z)|=|d|$, which is strict for $\tau>0$ because $d+\tau Z$ takes both signs with positive probability.
\end{proof}

Each term of \eqref{eq:finalstage} is largest for partners at the same position and decays strictly to zero as they separate, so the final-stage gain comes from close partners of opposite sign. The gap $g_\tau(d)-|d|$ is even in $d$, and by \eqref{eq:folded}, for $d\ge0$ and $\tau>0$,
\begin{equation}
 g_\tau(d)-d=2\tau\phi(d/\tau)-2d\,\Phi(-d/\tau),
 \qquad
 \partial_d\bigl(g_\tau(d)-d\bigr)
 =2\Phi(d/\tau)-2<0,
 \label{eq:gapdecay}
\end{equation}
because the terms containing $\phi$ cancel. The gap therefore decreases strictly in $|d|$, from its value at $d=0$ to its limit,
\[
 g_\tau(0)=\tau\sqrt{2/\pi},
 \qquad
 \lim_{|d|\to\infty}\bigl(g_\tau(d)-|d|\bigr)=0,
\]
and with $\tau=2\sigma$ the bracket in \eqref{eq:finalstage} equals $2\sigma\sqrt{2/\pi}$ at $d_i=0$.

The law of total variance turns this conditional gain into a reduction of the terminal variance.

\begin{corollary}[final-stage switching]
\label{cor:finalonly}
For any flux, the coupling that switches at the final stage alone, after reflection at every earlier stage, lowers the terminal variance of reflected pairing, integrated over $\R$, by exactly $\E\,\Gamma$.
\end{corollary}

\begin{proof}
This coupling coincides with reflected pairing before the final stage, and both final rules give the same conditional mean. Hence, in the law of total variance
\[
 \Var\bar U(x)=\E\,\Var\bigl(\bar U(x)\mid\Fil_{K-1}\bigr)+\Var\,\E\bigl[\bar U(x)\mid\Fil_{K-1}\bigr]
\]
only the expected conditional variance depends on the final rule, and its difference integrates over $\R$ to $\E\,\Gamma$.
\end{proof}

\subsection{Stage-by-stage gain for the heat equation}
\label{sec:everystage}

Extending Proposition~\ref{prop:pairwise} to every stage uses the field's monotone response to each displacement (Section~\ref{sec:motivation}), which persists until the final time for the heat equation. With $f\equiv0$ the particles do not interact, and a particle's displacement over the stages after stage $k$ is Gaussian with standard deviation
\[
 s_k=\sigma\sqrt{K-k-1},
\]
so that $s_{K-1}=0$. Conditioning on its position $X_i$ after stage $k$ replaces its terminal indicator by
\[
 \E\bigl[\ind\{X_i+s_kG\le x\}\bigm|X_i\bigr]=\Phi\Bigl(\frac{x-X_i}{s_k}\Bigr),
\]
where $G\sim\mathcal N(0,1)$ is independent of $X_i$. When $s_k=0$ we use the indicator itself, and the argument of Proposition~\ref{prop:pairwise} applies to these smoothed indicators.

Let $M_k(x)=\E[\bar U(x)\mid\Fil_k]$ be the conditional mean of the terminal pair mean. The martingale decomposition in the proof below splits the terminal variance into stage contributions,
\[
 \Var\bar U(x)=\sum_{k=0}^{K-1}\E\,\mathcal V_k(x),\qquad \mathcal V_k(x)=\Var\bigl(M_{k+1}(x)\bigm|\Fil_k\bigr),
\]
where $\mathcal V_k(x)$ is the variance that the draw of stage $k$ creates in the conditional mean. Proposition~\ref{prop:everystage} shows that, given the paired state entering stage $k$ and rank matching, the sign rule minimizes each pair's contribution to $\mathcal V_k$, with a closed-form gain.

\begin{proposition}[stage-by-stage gain for the heat equation]
\label{prop:everystage}
Let $f\equiv0$ and consider admissible couplings with $\pi_k=\mathrm{id}$. Given $\Fil_k$, let $a_i$, $b_i$, $m^A_i$, $m^B_i$, and $d_i=a_i-b_i$ be as in Proposition~\ref{prop:finalstage} for the state entering stage $k$. Among all joint laws of $(Z^A_i,Z^B_i)$ at stage $k$ with standard normal marginal distributions, the displacements of distinct matched pairs being independent, the sign rule minimizes the contribution of every matched pair to $\mathcal V_k(x)$, simultaneously for all $x\in\R$. Its gain over reflected pairing at stage $k$, the integral over $x$ of $\mathcal V_k(x)$ under reflection at stage $k$ minus $\mathcal V_k(x)$ under the sign rule at stage $k$, is
\begin{equation}
 \Gamma_k=\frac14\!\!\sum_{i:\,m^A_im^B_i<0}\!\!|m^A_im^B_i|\Bigl[g_{\sqrt{4\sigma^2+2s_k^2}}(d_i)-g_{\sqrt2\,s_k}(d_i)\Bigr]\;\ge\;0,
 \label{eq:everystage}
\end{equation}
which decreases strictly in each $|d_i|$ with $m^A_im^B_i<0$ and equals \eqref{eq:finalstage} for $k=K-1$.
\end{proposition}

\begin{proof}
\textbf{Conditional mean.} When $f\equiv0$, the transport and velocity updates leave the positions unchanged and sorting only relabels them, so we follow each particle $p$ of $A$, with its fixed mass $m_p$, by identity. At stage $l$ it receives $\sigma\zeta_{l,p}$, where $\zeta_l=\Pi^A_lZ_l$ and $\Pi^A_l$ is the $\Fil_l$-measurable permutation matrix that gives each particle of $A$ the draw of its rank. By the proof of Proposition~\ref{prop:law}, $\zeta_l\sim\mathcal N(0,I_N)$ is independent of $\Fil_l$, which contains $\Fil_{k+1}$ and $\zeta_{k+1},\dots,\zeta_{l-1}$ when $l>k$. Given $\Fil_{k+1}$, the remaining displacements $\sigma\sum_{l>k}\zeta_{l,p}$ are therefore independent $\mathcal N(0,s_k^2)$ variables, and
\[
 \E\bigl[U^A(x)\mid\Fil_{k+1}\bigr]=u_-+\sum_p m_p\,\Phi\Bigl(\frac{x-X^A_p}{s_k}\Bigr),
\]
where $X^A_p$ is the position of particle $p$ after stage $k$, and for $k=K-1$ the sum is $U^A(x)$ itself. The same holds for $B$ with $\zeta_l=\Pi^B_lD_lZ_l$, where $\Pi^B_l$ does the same for $B$ and $D_l=\operatorname{diag}(\varepsilon_l)$.

\textbf{Decomposition.} Since $|\bar U|\le|u_-|+\sum_i|m_i|$, $(M_k)_{k=0}^K$ is a bounded martingale, $\E[M_{k+1}\mid\Fil_k]=M_k$, from $M_0=\E\bar U(x)$ to $M_K=\bar U(x)$. Its increments are therefore uncorrelated, so $\Var\bar U(x)$ is the sum of their second moments, and $\E[(M_{k+1}(x)-M_k(x))^2\mid\Fil_k]=\mathcal V_k(x)$ gives the stage decomposition.

\textbf{Optimality.} Given $\Fil_k$, index the particles by rank entering stage $k$. Then
\[
 \begin{aligned}
 M_{k+1}(x)&=u_-+\tfrac12\sum_i\bigl(m^A_i\psi^A_i+m^B_i\psi^B_i\bigr),\\
 \psi^A_i&=\Phi\Bigl(\frac{x-a_i-\sigma Z^A_i}{s_k}\Bigr),\qquad
 \psi^B_i=\Phi\Bigl(\frac{x-b_i-\sigma Z^B_i}{s_k}\Bigr),
 \end{aligned}
\]
with $Z^A_i=Z_{k,i}$ and $Z^B_i=\varepsilon_{k,i}Z_{k,i}$, and $\psi=\chi$ when $s_k=0$. Distinct pairs are conditionally independent, so $\mathcal V_k(x)$ is the sum over $i$ of \eqref{eq:pairvar} with $\chi$ replaced by $\psi$, and only $\operatorname{Cov}(\psi^A_i,\psi^B_i)$ depends on the coupling. Each factor $\psi^A_i$, $\psi^B_i$ averages over the remaining displacements of one particle, whatever coupling the later stages use. To write the product $\psi^A_i\psi^B_i$ as one probability, we represent the two averages by standard normal random variables $G$ and $G'$, independent of each other and of $Z^A_i$ and $Z^B_i$, and set
\[
 X=a_i+\sigma Z^A_i+s_kG,\qquad Y=b_i+\sigma Z^B_i+s_kG'.
\]
Then $\psi^A_i\psi^B_i=\Pr\{X\le x,\,Y\le x\mid Z^A_i,Z^B_i\}$, so
\[
 \begin{aligned}
 \operatorname{Cov}\bigl(\psi^A_i,\psi^B_i\bigr)&=\operatorname{Cov}\bigl(\ind\{X\le x\},\ind\{Y\le x\}\bigr),\\
 \E[\psi^A_i\psi^B_i]&=\E_{G,G'}\Pr\bigl\{\sigma Z^A_i\le x-a_i-s_kG,\ \sigma Z^B_i\le x-b_i-s_kG'\bigr\}.
 \end{aligned}
\]
For each $(G,G')$ the probability is the joint distribution function of $(\sigma Z^A_i,\sigma Z^B_i)$ at two thresholds, so it lies between the Fr{\'e}chet--Hoeffding bounds \eqref{eq:frechet}. Reflection and synchronization attain these bounds at all thresholds at once, so the covariance lies between its reflected and synchronized values at every $x$ and the proof of Proposition~\ref{prop:pairwise} applies.

\textbf{Gain.} Each term of $\mathcal V_k$ is integrable in $x$, since $\Var\psi^A_i\le\mu(1-\mu)$ with $\mu=\Phi\bigl((x-a_i)/(\sigma^2+s_k^2)^{1/2}\bigr)$. By \eqref{eq:covdist}, $\int_\R\operatorname{Cov}(\psi^A_i,\psi^B_i)\,\mathrm dx=\tfrac12(\E|X-Y'|-\E|X-Y|)$, where $\E|X-Y'|$ depends only on the marginal distributions. The difference $X-Y$ is Gaussian under both rules,
\[
 \begin{aligned}
 X-Y&=d_i+2\sigma Z^A_i+s_k(G-G') &&\text{under reflection},\\
     &\qquad\sim\mathcal N\bigl(d_i,4\sigma^2+2s_k^2\bigr),\\
 X-Y&=d_i+s_k(G-G') &&\text{under synchronization},\\
     &\qquad\sim\mathcal N\bigl(d_i,2s_k^2\bigr).
 \end{aligned}
\]
Pairs of equal sign drop out. Since $\E|X-Y|=g_\tau(d_i)$ when $X-Y\sim\mathcal N(d_i,\tau^2)$, the weight $\tfrac12m^A_im^B_i$ of the covariance in $\mathcal V_k$ and the factor $\tfrac12$ in \eqref{eq:covdist} give \eqref{eq:everystage}. Let $\tau_1=(4\sigma^2+2s_k^2)^{1/2}>\tau_2=\sqrt2\,s_k$. The bracket is positive because $\partial_\tau g_\tau(d)=2\phi(d/\tau)>0$, so $\Gamma_k>0$ whenever partners of opposite sign enter stage $k$. For $d>0$ the $d$-derivative of the bracket, $2\Phi(d/\tau_1)-2\Phi(d/\tau_2)$ by \eqref{eq:gapdecay}, is negative. For $k=K-1$, $s_k=0$ and $g_0(d)=|d|$ reduce \eqref{eq:everystage} to \eqref{eq:finalstage}.
\end{proof}

Proposition~\ref{prop:everystage} compares couplings from the same paired state entering each stage, a state whose law depends on earlier couplings, and Example~3 measures the terminal variances of complete simulations (Section~\ref{sec:structure}).

\begin{remark}[loss of monotonicity under Burgers' flux]
\label{rem:reversal}
For $f(u)=u^2/2$ reordering can reverse the monotone response that the stage-by-stage analysis uses. Take two particles with masses $(1/2,-1/2)$ at initial positions $(\xi,0)$, $u_-=0$, $h=0.005$, $\nu=10^{-10}$, and two steps, and compare
\[
 \xi_1=-0.00275\qquad\text{and}\qquad\xi_2=-0.00225.
\]

Without diffusion, the field at $x_*=0.00125$ rises from $-1/2$ to $0$ as $\xi$ increases from $\xi_1$ to $\xi_2$. Both states start with velocities $(1/2,0)$, and the first transport moves the positive particle past the negative one for $\xi_2$ alone.

Re-sorting keeps the velocities $(1/2,0)$ for $\xi_1$. For $\xi_2$ the negative particle now ranks first, so it receives the velocity $-1/2$ and the positive particle receives $0$. With the positive particle listed first, the two configurations evolve as
\[
 \begin{array}{@{}l@{\quad}c@{\quad}c@{\quad}c@{\quad}c@{}}
  & \text{first transport} & \text{updated velocities} & \text{second transport} & U(x_*)\\[2pt]
  \xi_1 & (-0.00025,\,0) & (1/2,\,0) & (0.00225,\,0) & -1/2\\
  \xi_2 & (0.00025,\,0) & (0,\,-1/2) & (0.00025,\,-0.0025) & 0
 \end{array}
\]
For $\xi=-0.01$ the positions after the second transport are $(-0.005,0)$, with field $0$ at $x_*$, so the field falls from $\xi=-0.01$ to $\xi_1$.

With the displacements restored and $\sigma=10^{-6}$, on the event that the four displacements of a simulation have modulus at most $10^{-4}$, the second transport uses the velocities found without diffusion, and each terminal position shifts by at most $2\cdot10^{-4}$. Every terminal position without diffusion lies at least $0.001$ from $x_*$, so no particle crosses $x_*$ and the field values above hold on this event. Its complement has probability at most $\theta=8\Phi(-100)<\exp(-5003)$ and every field value lies in $[-1/2,1/2]$, so the terminal field $U(x_*,2h;\xi)$ started from $\xi$ satisfies $\E\,U(x_*,2h;\xi_1)\le-1/2+\theta$ and $\E\,U(x_*,2h;\xi_2)\ge-\theta/2$. Subtracting gives
\[
 \E\,U(x_*,2h;\xi_2)-\E\,U(x_*,2h;\xi_1)\ge\tfrac12-\tfrac32\theta>0 .
\]
Moving the positive mass to the right raises the expected terminal field, reversing the response that Proposition~\ref{prop:everystage} uses.
\end{remark}

\subsection{Matching and partner distance}
\label{sec:matching}

The distance dependence of the per-pair gains motivates comparing the total partner distance of the two matchings.

\begin{remark}[rank matching minimizes total partner distance]
\label{rem:sorted}
For real numbers $a_1\le\dots\le a_N$ and $b_1\le\dots\le b_N$, the identity matching minimizes $\sum_j|a_{\pi(j)}-b_j|$ over all bijections $\pi$. This is the monotone rearrangement of one-dimensional optimal transport \citep[Section~2.2, Theorem~2.18 and Remarks~2.19]{villani2003}, and here it follows by removing crossed pairs with
\[
 |a-b|+|a'-b'|\le |a-b'|+|a'-b|,
 \qquad a\le a',\quad b\le b'.
\]
Rank matching is this identity matching, and within-sign matching, with bijection $\pi_{\wit}$, restricts the allowed partners. Consequently,
\begin{equation}
  \sum_j\bigl|a_j-b_j\bigr|\;\le\;\sum_j\bigl|a_{\pi_{\wit}(j)}-b_j\bigr|
  \label{eq:sortedopt}
\end{equation}
for every pair of configurations, including ties.
\end{remark}

Rank and within-sign matching form different pairs, so we measure each pair against a common reference.

\begin{corollary}[reduction and partner distance]
\label{cor:distance}
Fix the paired state entering a diffusion stage, and let a matched pair have masses $m^A$ and $m^B$ and separation $d$. Relative to independent displacements for the two partners, the conditional variance of $\bar U$ just after the stage, integrated over $x$, is reduced by
\[
 \begin{aligned}
 &\tfrac14|m^Am^B|\,\bigl[g_{2\sigma}(d)-g_{\sqrt2\sigma}(d)\bigr]\\
 &\qquad\text{for a reflected pair of equal signs},\\[3pt]
 &\tfrac14|m^Am^B|\,\bigl[g_{\sqrt2\sigma}(d)-|d|\bigr]\\
 &\qquad\text{for a synchronized pair of opposite signs}.
 \end{aligned}
\]
Both reductions decrease as $|d|$ grows.
\end{corollary}

\begin{proof}
The proof of Proposition~\ref{prop:finalstage} applies with independent displacements as the reference, under which the indicators are uncorrelated and the expected partner distance after the stage is $g_{\sqrt2\sigma}(d)$. Both reductions decrease in $|d|$ by evenness and \eqref{eq:gapdecay} with $\tau=\sqrt2\sigma$, together with
\[
 \partial_d\bigl(g_{2\sigma}(d)-g_{\sqrt2\sigma}(d)\bigr)
 =2\Phi\bigl(d/(2\sigma)\bigr)-2\Phi\bigl(d/(\sqrt2\sigma)\bigr)<0,
 \qquad d>0.
 \qedhere
\]
\end{proof}

Close partners give the largest reductions whatever their signs, and switching retains rank matching. In Example~4, rank matching gives shorter partner distances than within-sign matching on the same paired state, and switching gives lower terminal variances than within-sign pairing over complete simulations (Section~\ref{sec:profiles}).

\section{Numerical experiments}
\label{sec:numerics}

Sections~\ref{sec:coupling} and~\ref{sec:analysis} compare the variance of couplings one diffusion stage at a time. Here we measure the variance reduction of switching in complete simulations and the execution time it saves at a common accuracy. Switching gives the lowest variance in Example~1 and the shortest measured execution time at each error {\ed tolerance} in Example~2. Example~3 relates the reduction to partners of opposite sign, and Example~4 relates it to close partners, consistent with Section~\ref{sec:analysis}. Example~5 combines switching with conditional evaluation and control variates, and Example~6 applies it to nonlinear observables.

\subsection{Setup}
\label{sec:setup}

Unless stated otherwise, we solve Burgers' equation, with flux $f(u)=u^2/2$, $\nu=0.1$, $T=1$, and $h=0.005$, evaluate fields at the $L=400$ points of $\Omega=[-5,5]$ (Section~\ref{sec:estimators}), and start from the two-pulse profile
\begin{equation}
 u_0(x)=0.8\,e^{-(x+0.7)^2/(2\cdot0.5^2)}
       -0.5\,e^{-(x-0.4)^2/(2\cdot0.3^2)},
 \label{eq:twopulse}
\end{equation}
whose derivative has neighboring positive and negative regions, so partners of opposite sign arise where these regions meet. The reference solution $u_{\mathrm{ref}}$ is a deterministic Fourier pseudospectral solution, checked against the Cole--Hopf formula \citep{hopf1950,cole1951} and, for the cubic flux, by refinement. One sign-switched pair mean with $N=8192$ follows $u_{\mathrm{ref}}$ closely at $T=1$ (Figure~\ref{fig:problems}, left), so the errors compared below are small relative to the field.

\begin{figure}[pos=!htbp]
\centering
\includegraphics[width=0.75\linewidth,height=0.43\textheight,keepaspectratio]{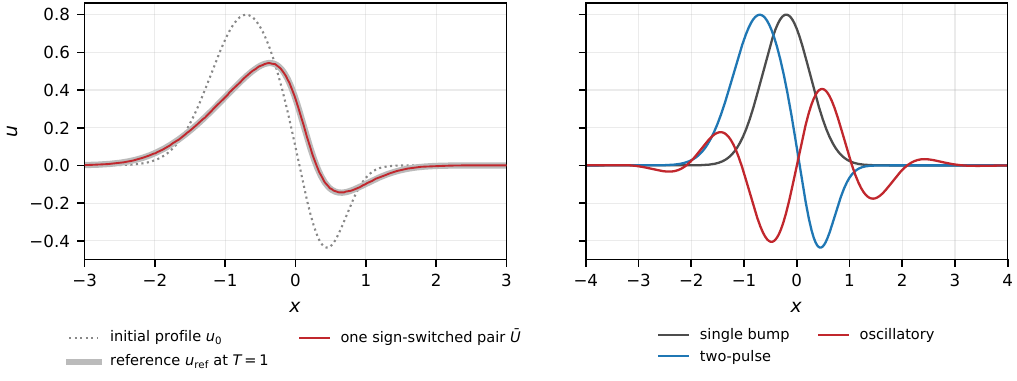}
\caption{Two-pulse Burgers, $N=8192$. Initial profile, reference solution at $T=1$, and one sign-switched pair mean (left). Example~4 initial profiles (right).}
\label{fig:problems}
\end{figure}

We compare two methods through ratios such as
\[
 \frac{V_{\swi}}{V_{\wit}}
 \qquad\text{and}\qquad
 \frac{\mathcal E_{\swi}(1)}{\mathcal E_{\wit}(1)}
\]
of the integrated variances and MSEs \eqref{eq:decomposition} of one realization, and a ratio below $1$ favors the method named first. Brackets give approximate $95\%$ percentile bootstrap intervals unless stated otherwise, and different methods use independent random inputs except in Example~5. All computations use Python with NumPy and SciPy \citep{abkindaripa2026software}.

\subsection{Example 1: variance reduction beyond within-sign pairing}
\label{sec:test1}

By Proposition~\ref{prop:law} and Corollary~\ref{cor:bias}, the pairings differ in integrated MSE only through the variance of the pair mean. Table~\ref{tab:headline} compares this variance for reflected, within-sign (Section~\ref{sec:rules}), and sign-switched pairing at $N=8192$, with $48$ realizations per method. Under reflected pairing, simulation $A$ has the law of one independent simulation, so its integrated variance, $2.07\times10^{-4}$, divided by two gives the entry for two independent simulations.

\begin{table}[pos=!htbp]
\centering
\caption{Example~1, two-pulse Burgers, $N=8192$. Integrated variance of the mean of two simulations and ratios between methods.}
\label{tab:headline}
\small
\setlength{\tabcolsep}{4pt}
\begin{tabularx}{\linewidth}{@{}>{\raggedright\arraybackslash}Xcc@{}}
\toprule
method & \shortstack{integrated\\variance $V$} & \shortstack{$V$ relative to two\\independent simulations}\\
\midrule
two independent simulations & $1.04\times10^{-4}$ & $1$\\
reflected pairing           & $4.33\times10^{-5}$ & $0.417$\\
within-sign pairing         & $7.43\times10^{-6}$ & $0.0717$\\
sign-switched pairing       & $5.49\times10^{-6}$ & $0.0530$\\
\midrule
comparison & \shortstack{integrated\\variance ratio} & \shortstack{integrated\\MSE ratio}\\
\midrule
within-sign $/$ reflected      & 0.172 $[0.145,\,0.204]$ & 0.206 $[0.174,\,0.244]$\\
sign-switched $/$ reflected    & 0.127 $[0.107,\,0.151]$ & 0.168 $[0.143,\,0.201]$\\
sign-switched $/$ within-sign & 0.739 $[0.682,\,0.801]$ & 0.819 $[0.754,\,0.893]$\\
\bottomrule
\end{tabularx}
\end{table}

Relative to two independent simulations, reflected pairing removes more than half of the integrated variance, while switching removes about $95\%$ of it. Switching and within-sign pairing, which both avoid reflecting partners of opposite sign (Proposition~\ref{prop:pairwise}), reduce the variance of reflected pairing six- to eightfold, and switching reduces that of within-sign pairing by $26\%$. The integrated-MSE ratios follow the same order and lie closer to $1$, because both errors include the common squared bias.

Under sign-switched pairing (Figure~\ref{fig:fields}, middle), the errors of the two simulations are nearly opposite and largely cancel in their mean. In the pair shown under reflected pairing (top), they agree near the center, where partners of opposite sign arise, and their mean keeps most of the error there.

\begin{figure}[pos=!htbp]
\centering
\includegraphics[width=0.60\linewidth,height=0.50\textheight,keepaspectratio]{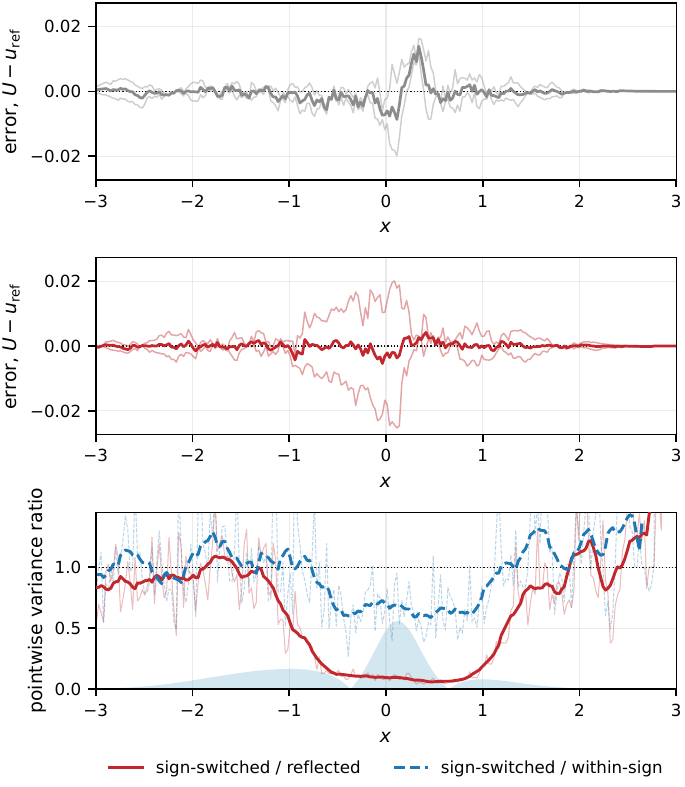}
\caption{Example~1, two-pulse Burgers, $N=8192$. Errors of the two simulations (thin curves) and of their mean (thick curves) for one realization of reflected pairing (top) and of sign-switched pairing (middle). Pointwise variance ratios (bottom), with $13$-point running means (thick curves) and the rescaled $|\partial_xu_{\mathrm{ref}}|$ (shaded).}
\label{fig:fields}
\end{figure}

The pointwise variance ratios (bottom) are smallest around the peak and trough, where the derivative changes sign and particles of opposite sign meet, and across the steep front between them. On $|x|\le1$, which carries $78\%$ of the integrated variance of within-sign pairing, switching reduces the variance by $34\%$, more than the $26\%$ over all of $\Omega$. On the flanks $|x|>1$, the pointwise ratios to within-sign pairing scatter around $1$, and the integrated ratio is $1.01$.

Averaging $R$ realizations takes work $W=RC$, where $C$ is the execution time of one realization. At equal work the product of variance and cost ranks estimators with a common mean \citep[Section~1.1.3]{glasserman2004}. The three pairings have comparable costs per realization, so the variance reduction of switching carries over to execution time, and its variance-times-cost ratio to within-sign pairing
\[
 \frac{V_{\swi}C_{\swi}}{V_{\wit}C_{\wit}}
\]
is $0.697$, about $30\%$ less variance for the same execution time.

\subsection{\edtitle{Example 2: execution time at an error tolerance}}
\label{sec:work}

Example~2 measures the execution time of each method at a common {\ed tolerance} $\mathrm{tol}$ for its integrated MSE \eqref{eq:decomposition}. The \emph{plan} of a method averages $R$ realizations, $R$ being the smallest positive integer with
\begin{equation}
 b^2_{\mathrm{hi}}+\frac{V_{\mathrm{hi}}}{R}\le\frac{\mathrm{tol}}{\kappa},
 \qquad \kappa=1.3,
 \label{eq:plan}
\end{equation}
where $V_{\mathrm{hi}}$ and $b^2_{\mathrm{hi}}$ are upper bootstrap confidence bounds for $V$ and $b^2$ from a pilot sample of $32$ realizations per method and $\kappa$ is a safety factor. Independent sampling and the three pairings share one squared-bias estimate, so their plans differ only through $V_{\mathrm{hi}}$.

We evaluate each plan on $32$ fresh, independent repetitions of the complete estimate and report the execution time per repetition. The counts and times reported below exclude the pilot sample, and the {\ed tolerance} was set from an earlier pilot study. A plan meets the {\ed tolerance} when the upper endpoint of an approximate simultaneous $95\%$ $t$ interval for its integrated MSE, Bonferroni-adjusted across the methods of that experiment, is below $\mathrm{tol}$.

\paragraph{Experiment~I.} At $N=8192$ and the {\ed tolerance} $7\times10^{-6}$ we compare switching with independent sampling, reflected pairing, and a randomized quasi-Monte Carlo (RQMC) adaptation of L{\'e}cot's quasi-random walk \citep[Eq.~(45)]{lecot_rims2001}, which spreads the diffusion inputs within each simulation using scrambled Sobol' points. The RQMC method, which we implemented for this comparison, shares our transport update but requires a separate bias estimate for its changed displacement law. All four plans \eqref{eq:plan} meet the {\ed tolerance} (Table~\ref{tab:heldout}). Switching uses $2$ pairs, that is $4$ simulations, where reflected pairing uses $34$, the RQMC method uses $46$, and independent sampling uses $65$. The sign-switched plan takes about an eighth of the time of reflected pairing. The RQMC plan clears the {\ed tolerance} by the widest margin, at an execution time close to independent sampling.

\begin{table}[pos=!htbp]
\centering
\caption{Example~2, two-pulse Burgers. Plans and their evaluation on fresh repetitions.}
\label{tab:heldout}
\small
\setlength{\tabcolsep}{3pt}
\begin{tabular*}{\linewidth}{@{\extracolsep{\fill}}lccccc@{}}
\toprule
method & $R$ & simulations & \shortstack{estimated\\integrated\\MSE} & \shortstack{upper\\confidence\\bound} & \shortstack{execution\\time (s)}\\
\midrule
\multicolumn{6}{l}{Experiment~I, {\ed tolerance} $7\times10^{-6}$, $N=8192$}\\
independent sampling & 65 & 65 & $5.35\times10^{-6}$ & $6.74\times10^{-6}$ & 6.66\\
reflected pairing        & 17 & 34 & $5.03\times10^{-6}$ & $6.36\times10^{-6}$ & 3.34\\
RQMC method          & 46 & 46 & $4.02\times10^{-6}$ & $4.87\times10^{-6}$ & 6.85\\
sign-switched pairing & 2 & 4 & $5.10\times10^{-6}$ & $5.61\times10^{-6}$ & 0.407\\
\midrule
\multicolumn{6}{l}{Experiment~II, {\ed tolerance} $2.5\times10^{-5}$, $N=2048$}\\
reflected pairing        & 18 & 36 & $1.40\times10^{-5}$ & $1.83\times10^{-5}$ & 0.615\\
within-sign pairing      & 3 & 6 & $1.77\times10^{-5}$ & $1.94\times10^{-5}$ & 0.119\\
sign-switched pairing & 2 & 4 & $1.98\times10^{-5}$ & $2.20\times10^{-5}$ & 0.0715\\
\bottomrule
\end{tabular*}
\end{table}

\paragraph{Experiment~II.} At $N=2048$ and the {\ed tolerance} $2.5\times10^{-5}$ we compare switching with reflected pairing and with within-sign pairing, its closest competitor in Example~1. All three plans meet the {\ed tolerance}, and switching uses $2$ pairs where within-sign pairing uses $3$ slightly costlier pairs, reducing the measured execution time by about $40\%$. In both experiments switching gives the fastest evaluated plan that meets the {\ed tolerance}.

\paragraph{Bias and work.} Switching reduces the MSE most while $V/R$ exceeds $b^2$. By Corollary~\ref{cor:bias}, at equal $R$
\begin{equation}
 \frac{\mathcal E_{\swi}(R)}{\mathcal E_{\refl}(R)}
 =\frac{b^2+V_{\swi}/R}{b^2+V_{\refl}/R}\longrightarrow1
 \qquad\text{as }R\to\infty
 \label{eq:crossover}
\end{equation}
whenever $b^2>0$. The left panel of Figure~\ref{fig:cost} compares the evaluated plans of Experiment~I (diamonds) with the model curves $b^2+V/R$ against the execution time $RC$.

\begin{figure}[pos=!htbp]
\centering
\includegraphics[width=0.75\linewidth,height=0.43\textheight,keepaspectratio]{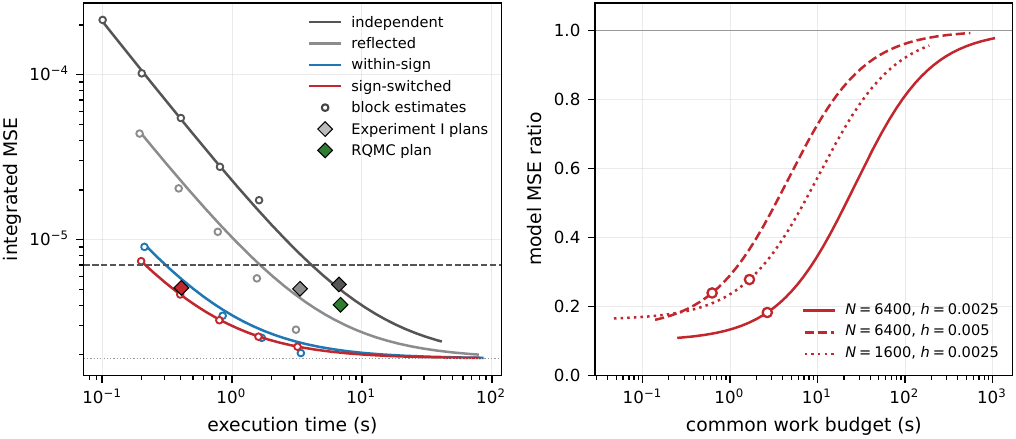}
\caption{Example~2, two-pulse Burgers. Integrated MSE against execution time, with the {\ed tolerance} dashed and $b^2$ dotted (left), and model MSE ratio of sign-switched to reflected pairing at equal work (right).}
\label{fig:cost}
\end{figure}

The curves use $b^2$, $V$, and $C$ estimated from the $48$ realizations per method of Example~1, and the evaluated plans and averages over disjoint groups of these realizations (circles) lie close to them. The RQMC plan has no curve because its bias is estimated separately. The sign-switched curve crosses the {\ed tolerance} first, at a small fraction of the time of the reflected and independent curves, and all curves approach the common squared bias $b^2\approx1.9\times10^{-6}$, as \eqref{eq:crossover} predicts.

At equal work $W$, a method $\bullet$ whose realizations cost $C_\bullet$ averages $R_\bullet=W/C_\bullet$ realizations, treating $R$ as continuous, and the ratio in \eqref{eq:crossover} becomes
\begin{equation}
 \frac{b^2+V_{\swi}C_{\swi}/W}{b^2+V_{\refl}C_{\refl}/W},
 \label{eq:equalwork}
\end{equation}
which is close to the variance-times-cost ratio for small $W$ and rises toward $1$ as $W$ grows. The right panel of Figure~\ref{fig:cost} plots \eqref{eq:equalwork} for three combinations of $N$ and $h$, with $b^2$, $V$, and $C$ estimated from $64$ realizations per method. At $W=C_{\swi}V_{\swi}/b^2$ (circles), where the variance term of switching equals the squared bias, the ratios are still between $0.18$ and $0.28$. Once the squared bias dominates, a larger $N$ or a smaller $h$ reduces it and, at these settings, also reduces the variance ratio of switching to reflected pairing, to $0.0986$ at $N=6400$ and $h=0.0025$.

\subsection{Example 3: particle signs and positions}
\label{sec:structure}

Lemma~\ref{lem:conditional} and Propositions~\ref{prop:finalstage} and~\ref{prop:everystage} show that, at a fixed paired state, the gain comes from partners of opposite sign and is largest when they are close. Example~3 tests this on complete simulations by comparing sign-switched with reflected pairing under heat and Burgers dynamics (Table~\ref{tab:structure}), with $48$ realizations per method. Its four configurations place $N=2048$ particles on the scale of the diffusion length $\sqrt{2\nu T}\approx0.45$. Their masses $m_i$ and particle intervals are
\[
 \begin{array}{@{}l@{\quad}l@{\quad}l@{}}
 \text{alternating signs} & \pm1/N \text{ of alternating sign} & [-0.25,0.25],\\[2pt]
 \text{merging blocks} & 1/N \text{ and } -1/N & [-0.9,-0.5] \text{ and } [0.5,0.9],\\[2pt]
 \text{separated blocks} & 1/N \text{ and } -1/N & [-2.0,-1.4] \text{ and } [1.4,2.0],\\[2pt]
 \text{one sign} & -2/N & [-1,1].
 \end{array}
\]
Each block holds $N/2$ particles, the particles in each interval are equally spaced including its endpoints, and $u_-=0$ throughout. Merging blocks leave a gap of $1.0$, across which particles of opposite sign meet, and separated blocks a gap of $2.8$, across which they rarely meet.

\begin{table}[pos=!htbp]
\centering
\caption{Example~3, heat and Burgers, $N=2048$. Variance ratio, sign-switched to reflected pairing, by configuration.}
\label{tab:structure}
\small
\begin{tabular}{lcc}
\toprule
configuration & heat & Burgers\\
\midrule
alternating signs & 0.0492 $[0.0405,0.0608]$ & 0.0625 $[0.0516,0.0766]$\\
merging blocks    & 0.476 $[0.421,0.545]$   & 0.481 $[0.426,0.544]$\\
separated blocks  & 0.994 $[0.917,1.08]$    & 1.01 $[0.930,1.10]$\\
one sign          & 0.987 $[0.902,1.08]$    & 0.967 $[0.822,1.17]$\\
\bottomrule
\end{tabular}
\end{table}

The measured ratios agree with this. With alternating signs, switching leaves about $5\%$--$6\%$ of the variance under reflected pairing, and with merging blocks about half of it. For separated blocks the ratios are consistent with $1$, since the wide gap leaves few partners of opposite sign before $T=1$. With one sign the two couplings are the same algorithm (Lemma~\ref{lem:conditional}), so the exact ratio is $1$ and the estimates differ from it by statistical error alone.

The Burgers ratios closely follow the heat ratios, for which, given the paired state entering a stage and rank matching, the sign rule minimizes each matched pair's share of that stage's contribution to the terminal variance (Proposition~\ref{prop:everystage}). For alternating signs the nearly canceling masses keep the field, and hence the Burgers transport velocity, small, so diffusion dominates. For merging blocks the field reaches $0.5$ and transport matters, yet the ratios again agree, so close partners of opposite sign help with and without transport.

\subsection{Example 4: profiles, times, viscosities, and fluxes}
\label{sec:profiles}

Rank matching minimizes the total partner distance (Remark~\ref{rem:sorted}), and closer partners give larger per-pair reductions (Corollary~\ref{cor:distance}). Example~4 tests whether these results for a fixed paired state carry over to complete simulations, comparing partner distances and terminal variances for the two-pulse profile, a single bump, and an oscillatory profile (Figure~\ref{fig:problems}, right),
\begin{equation}
 \begin{aligned}
 u_{0,\mathrm{bump}}(x)&=0.8\exp\!\left(-\frac{(x+0.2)^2}{2\cdot0.45^2}\right),\\
 u_{0,\mathrm{osc}}(x)&=0.45\sin(3x)\exp\!\left(-\frac{x^2}{2\cdot1.1^2}\right).
 \end{aligned}
 \label{eq:profiles}
\end{equation}
The single bump has one change of gradient sign and the other two have more, so the three range from few to many places where particles of opposite sign meet. We evolve each to $T=0.25$, $1$, and $2$ with $N=2048$ and $48$ realizations per method. To isolate the matching, we apply both matchings to the paired state that switching generates just before the final stage, in $24$ separate realizations of each profile and time, and compute for each state the partner-distance ratio
\[
 \frac{\sum_{j=1}^N\bigl|a_j-b_j\bigr|}{\sum_{j=1}^N\bigl|a_{\pi_{\wit}(j)}-b_j\bigr|}\;\le\;1,
\]
where $a_j$ and $b_j$ are the positions of rank $j$ in $A$ and $B$ and the inequality is \eqref{eq:sortedopt}. Figure~\ref{fig:profiles} (right) reports the mean of this ratio over the $24$ states.

\begin{figure}[pos=!htbp]
\centering
\includegraphics[width=0.70\linewidth,height=0.43\textheight,keepaspectratio]{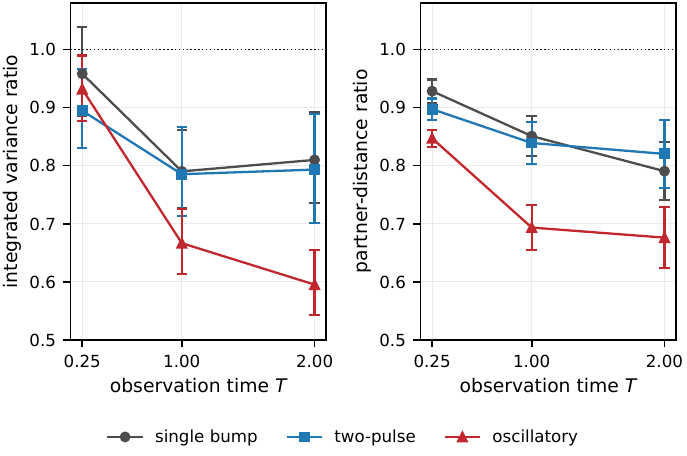}
\caption{Example~4, Burgers, $N=2048$. Variance ratio, sign-switched to within-sign pairing, with approximate $95\%$ bootstrap intervals (left), and mean partner-distance ratio with two-standard-error bars (right). In the right panel, both matchings act on the same paired states from switching.}
\label{fig:profiles}
\end{figure}

\begin{figure}[pos=!htbp]
\centering
\includegraphics[width=0.55\linewidth,height=0.43\textheight,keepaspectratio]{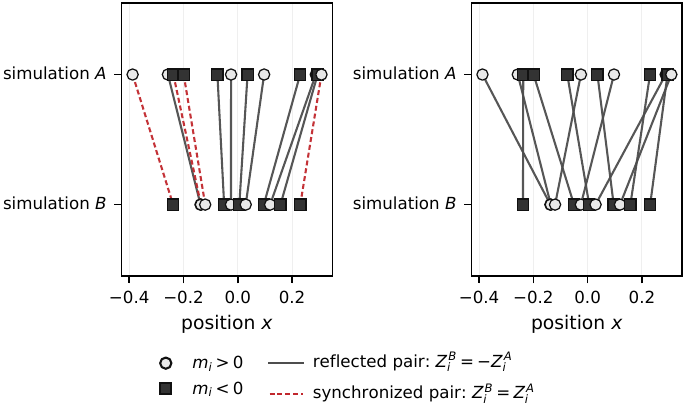}
\caption{Example~4. Rank matching (left) and within-sign matching (right) of one paired $12$-particle state.}
\label{fig:construction}
\end{figure}

\paragraph{Profiles and times.} In Figure~\ref{fig:profiles} (left), the variance ratios of sign-switched to within-sign pairing lie between $0.895$ and $0.958$ at $T=0.25$, near $1$. By $T=1$ they fall to between $0.667$ and $0.790$, and at $T=2$ lie between $0.595$ and $0.810$, with the oscillatory profile lowest at both times.

\paragraph{Partner distance.} Figure~\ref{fig:construction} shows both matchings for a $12$-particle version of the alternating configuration of Example~3 just before the final diffusion stage. Rank matching (left) gives close partners and synchronized pairs of opposite sign, and within-sign matching (right) the long pairs described in Section~\ref{sec:rules}.

The distance ratios range from $0.676$ to $0.928$ (Figure~\ref{fig:profiles}, right), and the oscillatory profile has both the largest reduction in distance and the smallest variance ratio at later times. For each profile the distance ratio, like the variance ratio, is closest to $1$ at $T=0.25$, consistent with the per-pair reductions in Corollary~\ref{cor:distance}.

\paragraph{Viscosity.} For the two-pulse profile at $T=1$, with $N=2048$ and $48$ realizations per method, the variance ratio of sign-switched to reflected pairing falls from $0.343$ to $0.0914$ as $\nu$ increases through $\{0.02,\allowbreak0.05,\allowbreak0.1,\allowbreak0.2,\allowbreak0.5\}$. The ratio to within-sign pairing lies between $0.693$ and $0.915$, with its largest value at $\nu=0.02$, where transport dominates diffusion and the ratio is consistent with $1$ at this sample size.

\paragraph{Cubic flux.} The nonconvex flux $f(u)=u^3/3$ has the nonnegative characteristic speed $f'(u)=u^2$, so all transport is to the right. In the two-pulse problem with $N=1600$ (Section~\ref{sec:test2}, Figure~\ref{fig:controls}), the ratio of switching to within-sign pairing is $0.769$, against $0.792$ for Burgers' equation.

Overall, the reduction relative to within-sign pairing persists across profiles, later times, viscosities from $0.05$ to $0.5$, and both fluxes.

\subsection{Example 5: conditional evaluation and control variates}
\label{sec:test2}

Example~5 tests whether two standard variance-reduction techniques, applied to reflected pairing, reach the variance of switching, and whether they add to it. The first is \emph{conditional evaluation} of the final diffusion stage, a form of conditional Monte Carlo \citep[Section~6.4, Eqs.~(6.53)--(6.55)]{glasserman2004}, which replaces the field of each simulation by its conditional expectation given the positions $a_i$ and masses $m_i$ entering the final stage,
\[
 \E\bigl[U(x,T)\bigm|\Fil_{K-1}\bigr]
 =u_-+\sum_{i=1}^N m_i\,\Phi\!\left(\frac{x-a_i}{\sigma}\right),
\]
removing the randomness of the final displacements analytically.

A control variate subtracts from $\bar U$ a multiple of a correlated field with known mean \citep[Section~4.1]{glasserman2004}. Our two controls label each particle by its initial rank $i$ and follow it through the sorts, with $X^0_i$ its initial position, $v^0_i$ its initial velocity, and $S_i$ the sum of the displacements it receives. The heat and velocity controls, which reuse these displacements, are the fields
\[
 \begin{aligned}
 U_{\mathrm{heat}}(x)&=u_-+\sum_{i=1}^N m_i\,\ind\{X^0_i+S_i\le x\},\\
 U_{\mathrm{vel}}(x)&=u_-+\sum_{i=1}^N m_i\,\ind\{X^0_i+v^0_iT+S_i\le x\}.
 \end{aligned}
\]
By the argument of Proposition~\ref{prop:law}, each $S_i$ is a sum of $K$ independent normal displacements with variance $\sigma^2$, so $S_i\sim\mathcal N(0,2\nu T)$ and both fields have known means. For a pair, the estimator with the heat control is $\bar U-c\,(\bar U_{\mathrm{heat}}-\E\bar U_{\mathrm{heat}})$, where $\bar U_{\mathrm{heat}}$ is the pair mean of $U_{\mathrm{heat}}$, and likewise for the velocity control. The coefficient $c$, one for the whole field, is estimated from a separate sample of $16$ pairs, so the estimator keeps the mean of $\bar U$.

We compare the estimators on the two-pulse problem with $N=1600$ for both fluxes, using $64$ realizations, the same random displacements for every estimator, and fields evaluated at $801$ points on $[-8,8]$. Control costs include coefficient fitting and known-mean evaluation, amortized over the $64$ realizations. In Figure~\ref{fig:controls}, each row marked $+$ adds one technique to the reflected or sign-switched pairing above it, separately from the other $+$ rows.

\begin{figure}[pos=!htbp]
\centering
\includegraphics[width=0.70\linewidth,height=0.43\textheight,keepaspectratio]{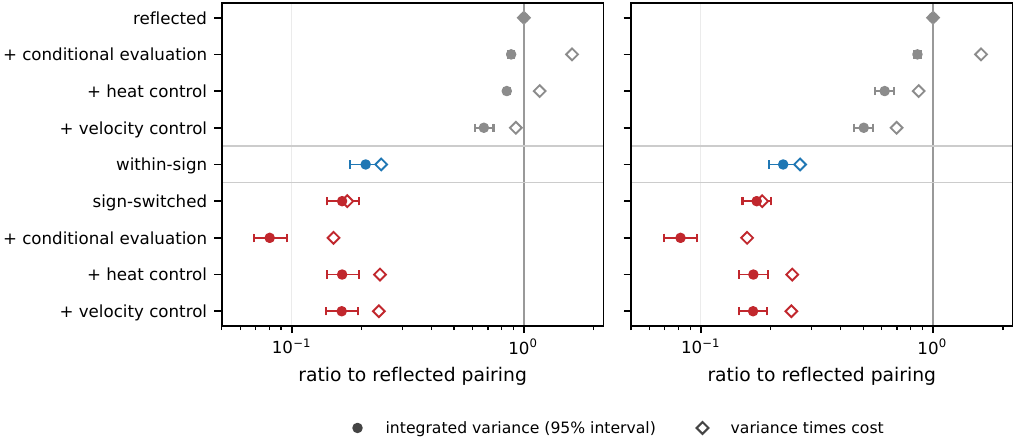}
\caption{Example~5, $N=1600$. Variance and variance times cost relative to reflected pairing, Burgers (left) and cubic (right) flux.}
\label{fig:controls}
\end{figure}

Figure~\ref{fig:controls} shows that the velocity control, the better of the two, removes about a third of the variance under reflected pairing for Burgers' equation (ratio $0.672$) and about half for the cubic flux ($0.503$), whereas within-sign and sign-switched pairing remove about four fifths of it. After switching, either control changes the variance by at most about $3.5\%$, and switching alone has the lower variance times cost.

Added to switching, conditional evaluation roughly halves the variance again for both fluxes and gives the lowest variance times cost in Figure~\ref{fig:controls}. For Burgers' equation, in units of the reflected variance, it removes about $0.12$ without switching and $0.08$ with switching. In this two-pulse Burgers problem, switching at the final stage alone (Corollary~\ref{cor:finalonly}) gives about $4\%$ of the reduction obtained by switching at every stage, so switching at the earlier stages, which conditional evaluation leaves unchanged, is essential.

Conditional evaluation sums $N$ terms at each evaluation point, so its cost grows with the number of points. Adding it to switching multiplies the variance times cost of switching alone by $0.613$ at $201$ points and by $1.49$ at $1601$ points, so the combination reduces the variance times cost for grids up to the $801$ points of Figure~\ref{fig:controls}.

\subsection{Example 6: peak gradient and level location}
\label{sec:observables}

Example~6 extends the comparison to nonlinear observables. A pair gives two estimators of an observable $Q$, a single number computed from the field. The mean of $Q$ is $\tfrac12\bigl(Q(U^A)+Q(U^B)\bigr)$ and $Q$ of the mean is $Q(\bar U)$. By Proposition~\ref{prop:law} the first has the same bias under every admissible coupling, whereas the bias of the second can depend on the coupling (Remark~\ref{rem:nonlinear}). We use the peak gradient and the level location, the first downward crossing of $u=0.25$ to the right of the global maximum,
\[
 \begin{aligned}
 Q_{\mathrm{peak}}(U)&=\max_{1\le\ell\le L}\bigl|D_\ell U\bigr|,\\
 Q_{\mathrm{level}}(U)&=\min\bigl\{x\ge x_{\ell^*}\,:\,\widetilde U(x)=0.25\bigr\},
 \end{aligned}
\]
where $D_\ell U$ is the central difference of $U$ at $x_\ell$, one-sided at the two ends of $\Omega$, the field is largest at $x_{\ell^*}$ among the points $x_1,\dots,x_L$ of Section~\ref{sec:estimators}, and $\widetilde U$ interpolates $U$ linearly between them.

We evaluate both observables on the $48$ realizations per method of Example~1, and Table~\ref{tab:nonlinear} compares their empirical MSEs (mean squared differences from $Q(u_{\mathrm{ref}})$) and gives the estimated squared-bias share of the sign-switched MSE.

\begin{table}[pos=!htbp]
\centering
\caption{Example~6, two-pulse Burgers, $N=8192$. Empirical MSE ratios for two observables and two estimators, and the estimated bias$^2$ share of the sign-switched MSE.}
\label{tab:nonlinear}
\small
\setlength{\tabcolsep}{3pt}
\begin{tabularx}{\linewidth}{@{}>{\raggedright\arraybackslash}X>{\raggedright\arraybackslash}Xccc@{}}
\toprule
observable & estimator & \shortstack{sign-switched $/$\\reflected} & \shortstack{sign-switched $/$\\within-sign} & \shortstack{estimated\\bias$^2$ share}\\
\midrule
peak gradient        & $Q$ of the mean      & 0.416 $[0.244,0.702]$ & 0.602 $[0.359,0.984]$ & 0.56\\
                     & mean of $Q$          & 0.832 $[0.603,1.14]$ & 1.06 $[0.803,1.39]$ & 0.87\\
level $x(u=0.25)$    & $Q$ of the mean      & 0.129 $[0.0699,0.247]$ & 0.855 $[0.497,1.42]$ & 0.19\\
                     & mean of $Q$          & 0.0898 $[0.0527,0.158]$ & 0.847 $[0.497,1.49]$ & 0.30\\
\bottomrule
\end{tabularx}
\end{table}

For the peak gradient, $Q$ of the mean has a lower MSE under switching than under either alternative. The mean of $Q$ benefits less, since an estimated $87\%$ of its MSE under switching is squared bias, which is the same for every admissible coupling. On average the peak gradient of one simulation exceeds the reference value $1.29$ by $0.10$ to $0.11$ under all three couplings, as the maximum of a noisy field tends to exceed the maximum of its mean.

Averaging the fields before taking the maximum reduces the fluctuations, and the coupling reduces them further, so $Q$ of the mean has less than half the MSE of the mean of $Q$ under all three couplings, and about a fifth under switching (ratio $0.43$ under reflected, $0.38$ under within-sign, and $0.21$ under sign-switched pairing).

For the level location, the ratios of switching to within-sign pairing are consistent with $1$ at this sample size, and switching reduces the MSE relative to reflected pairing for both estimators.

\section{Conclusion}
\label{sec:conclusions}

Sign-switched pairing reduces the computational cost of field estimation in the tested gradient random-walk problems by matching particles by position and choosing identical or opposite diffusion displacements according to the signs of their masses. In the two-pulse Burgers problem, switching removes about $95\%$ of the integrated variance of the mean of two independent simulations and reduces that of within-sign pairing by $26\%$ (Section~\ref{sec:test1}). At a common {\ed tolerance} for the integrated mean squared error, it takes about an eighth of the execution time of reflected pairing. At a separate {\ed tolerance}, it takes about $40\%$ less time than within-sign pairing, giving the fastest evaluated plan in both comparisons (Section~\ref{sec:work}). For configurations with particles of both signs, its variance ratio to reflected pairing under Burgers' equation nearly matches that under the heat equation (Section~\ref{sec:structure}), and its reduction relative to within-sign pairing persists across profiles at later times, viscosities from $0.05$ to $0.5$, and both nonlinear fluxes (Section~\ref{sec:profiles}). Because the pair mean keeps the bias of a single simulation, the saved execution time can go into more particles or a smaller time step once the squared bias dominates (Section~\ref{sec:work}).

Each simulation keeps exactly the law of an uncoupled simulation although partners and their signs change, so the pair mean has the bias of a single simulation (Proposition~\ref{prop:law} and Corollary~\ref{cor:bias}). For a fixed paired state and matching, the sign rule minimizes each matched pair's contribution to the variance of the pair mean just after a diffusion stage, a contribution that reflected pairing maximizes for partners of opposite sign (Proposition~\ref{prop:pairwise}). The final-stage gain over reflected pairing has a closed form for any flux (Proposition~\ref{prop:finalstage}), and for the heat equation the optimality and the closed form extend to each stage's contribution in a martingale decomposition of the terminal variance (Proposition~\ref{prop:everystage}).

Law preservation follows from the solver's independent Gaussian draws when the matching and the choice of reflection or synchronization are fixed before each stage's draws, and the sign rule requires only signed weights. For nonlinear fluxes, a stage-by-stage result must account for reordering, which can reverse the effect of a displacement on the terminal field (Remark~\ref{rem:reversal}). Relaxation approximations \citep{bertaglia2024} and the levels of multilevel Monte Carlo \citep[Section~2]{giles2008} are further settings where switching could reduce the variance. In two dimensions, a matching of minimal total partner distance, as rank matching is in one dimension (Remark~\ref{rem:sorted}), could replace rank matching. Random vortex methods \citep{chorin1973}, whose circulations are signed and whose displacements are Gaussian, are then a natural next test, and two-dimensional gradient random-walk methods \citep{shermanmascagni1994}, whose particles carry a direction as well as a mass, require a sign rule for that direction.

\section*{Code and data availability}
The solver, the simulation data, and the scripts and run settings that reproduce the tables and figures of this paper are archived as version~1.0 \citep{abkindaripa2026software} at \url{https://doi.org/10.5281/zenodo.23055907}.



\bibliographystyle{elsarticle-num}

\bibliography{references-coupling}

\end{document}